\documentclass[11pt]{article}

\usepackage[a4paper]{geometry}
\usepackage{graphicx}

\usepackage{amsmath,amsfonts}
\usepackage{bm}
\usepackage{cases}
\usepackage{amsthm}
\usepackage{algorithm}
\usepackage{algorithmicx}
\usepackage{algpseudocode}
\usepackage{soul}
\usepackage{enumerate}
\usepackage{longtable}
\usepackage{authblk}
\usepackage{wrapfig}

\usepackage{comment}
\usepackage{bm}
\usepackage{amssymb}
\usepackage{url}
\usepackage{type1cm}
\usepackage{cite}
\usepackage{here}
\usepackage{blkarray}
\usepackage{easy-todo}

\newcommand{\real}{\mathbb{R}}
\newcommand{\SolSet}{\mathcal{S}}

\newcommand{\Neighborhood}{\mathcal{N}}
\newcommand{\Order}{\mathcal{O}}
\providecommand{\abs}[1]{\left\lvert#1\right\rvert}
\providecommand{\norm}[1]{\left\|#1\right\|}

\newtheorem{assumption}{Assumption}[section]
\newtheorem{theorem}{Theorem}[section]
\newtheorem{lemma}{Lemma}[section]
\newtheorem{proposition}{Proposition}[section]
\newtheorem{corollary}{Corollary}[section]

\title{An Infeasible Interior-Point Arc-search Method with Nesterov's Restarting Strategy for Linear Programming Problems}
\date{2023/09/29}
\author{
  Einosuke Iida\thanks{Department of Mathematical and Computing Science,Tokyo Institute of Technology} and
  Makoto Yamashita\thanks{Department of Mathematical and Computing Science,Tokyo Institute of Technology. The research of Makoto Yamashita was partially supported by JSPS KAKENHI (Grant Number: 21K11767)}
}

\begin{document}

\maketitle

\begin{abstract}
  An arc-search interior-point method is a type of interior-point methods that approximates
  the central path by an ellipsoidal arc,
  and it can often reduce the number of iterations.
  In this work,
  to further reduce the number of iterations and computation time for solving linear programming problems,
  we propose two arc-search interior-point methods using
  Nesterov's restarting strategy that is well-known method to accelerate the gradient method with a momentum term.
  The first one generates a sequence of iterations in the neighborhood,
  and we prove that the convergence of the generated sequence to an optimal solution and the computation complexity is polynomial time.
  The second one incorporates the concept of the Mehrotra type interior-point method to improve numerical performance.
  The numerical experiments demonstrate that
  the second one reduced the number of iterations and computational time.
  In particular, the average number of iterations was reduced
  compared to existing interior-point methods
  due to the momentum term.
\end{abstract}

{\bf Keywords:} interior-point method, arc-search, Nesterov's restarting method, linear programming.

\section{Introduction}
\label{section_introduction}
Numerical methods for solving linear programming (LP) problems have been studied well in the literature,
since LP problems have wide-range applications including practical ones and theoretical ones.
In particular, the simplex method, the gradient method and the interior-point method are well-known methods.
The interior-point method was originally proposed by Karmarkar~\cite{karmarkar1984new} in 1984,
and it has been improved by,
for example, the primal-dual interior-point method~\cite{kojima1989primal}
and Mehrotra type predictor-corrector method~\cite{Mehrotra1992}.
For more details,
see Wright~\cite{wright1997primal} and references therein.

The computational efficiency of higher-order
algorithms~\cite{monteiro1990polynomial,Mehrotra1992,gondzio1996multiple,altman1999regularized,kheirfam2018polynomial,espaas2022interior}
using second-order or higher derivatives in framework of interior-point methods has been getting a lot of attention.
Mehrotra type predictor-corrector method~\cite{Mehrotra1992}
can be considered as one of the high-order algorithms.
Lustig et al.~\cite{lustig1992implementing} reported
that Mehrotra type predictor-corrector method was effective in reducing the computation time.
The method has also been implemented in software packages~\cite{yamashita2003implementation,yamashita2012latest}.
However, the higher-order algorithms sometimes had a worse polynomial bound,
or the analysis of computational complexity was not simple.

The arc-search interior-point method~\cite{yang2011polynomial} proposed for linear programming
also employs higher derivatives,
and Yang and Yamashita~\cite{yang2018arc}
proved that the number of iterations in the arc-search infeasible interior-point method can be bounded by $\Order(nL)$,
which is equal to the best known polynomial bound for the infeasible interior-point algorithms.
The arc-search interior-point method has been extended for
quadratic programming with box constraints~\cite{yang2013constrained},
second-order cone programming~\cite{yang2017arc},
semidefinite programming~\cite{zhang2019primal},
complementarity problem~\cite{yuan2018wide},
and general nonlinear programming~\cite{Yamashita2021}.

A main idea of the arc-search interior-point method
is to approximate the central path with an ellipsoidal arc
and this leads to a reduction of the number of iterations.
The central path is a smooth curve that converges to an optimal solution,
but standard interior-point methods use a linear approximation to trace
the central path and apply a line search.
Since the central path is a curve, it can be expected
that an ellipsoidal arc can approximate the central path better than
the line search with a fixed search direction.
In fact, Yang and Yamashita~\cite{yang2018arc} showed through numerical experiments that
the arc-search interior-point method reduces the number of iterations for solving LP test instances.

However, the computational time for each iteration
in the arc-search interior-point method tends to increase compared to line-search methods
due to calculations of the higher-order derivatives.
Therefore, reducing the number of iterations further
without calculating the higher-order derivatives
is one of important
issues to attain better computational time.

On the other hand, Nesterov's restarting strategy~\cite{nesterov1983} is widely employed to
improve the computation time of first-order methods.
Nesterov proved that this technique reduces the worst-case convergence rate of the gradient methods for minimizing smooth convex functions
from $\Order(1/k)$ to $\Order(1/k^2)$.
Dozat~\cite{timothy2016incorporating}
utilized this technique in a framework of stochastic gradient descent methods
in Deep Learning.

For LP problems,
Morshed and Noor-E-Alam~\cite{Morshed2020} examined a method that combines Nesterov's
restarting strategy with the affine scaling method and proved that their method is polynomial time.
In addition, their numerical results showed that their method reduces the computation time compared to a conventional affine scaling method.
To the best of the author's knowledge, however,
Nesterov's restarting strategy has not been introduced in the arc-search interior-point methods.

In this paper,
we propose a new arc-search method for LP problems combined with Nesterov's restarting strategy so that
the arc-search interior-point method can further reduce the number of iterations.
We establish the convergence of a generated sequence of the proposed method to an optimal solution,
and we also show that the proposed method is a polynomial-time algorithm.
Furthermore, to reduce actual computation time, we modify the proposed method
by introducing a concept of the Merhotra type interior-point method.
From numerical experiments on Netlib test instances~\cite{browne1995netlib},
we observed that this modified method achieved better performance than existing methods in terms of both the number of iterations and computation time.
In particular, the number of iterations is successfully reduced by 5.4\% compared to an existing arc-search method.

This paper is organized as follows.
Section~\ref{section_preliminaries} introduces the standard form of LP problems in this paper and gives a brief summary of arc-search interior-point methods.
We describe the proposed method in Section~\ref{section_proposed_method},
and we discuss its convergence properties in Section~\ref{section_theoretical_proof}.
In Section~\ref{section_numerical_experiments},
we modify the proposed method to improve performance in numerical experiments and compare it with existing methods.
Finally,
we give a conclusion of this paper and discuss future directions in Section~\ref{section_conclusion}.

\subsection{Notations}
We denote $i$-th element of a vector $x$ by $x_i$
and the Hadamard product (the element-wise product) of two vectors $x$ and $s$ by
$x \circ s$.
We use $\norm{x}$ and  $\norm{x}_{\infty}$
for the Euclidean norm and the infinity norm of a vector $x$, respectively.
The identity matrix and the vector of all ones of an appropriate dimension are denoted by $I$ and $e$, respectively.
We use a superscript $\top$ to denote the transpose of a vector or a matrix.
The block column vector $\left[x^\top, s^\top\right]^\top$
is shortly expressed by $(x, s)$.
For $x \in \real^n$, we will use the capital character $X \in \real^{n \times n}$
for the diagonal matrix that puts the components of $x$ at its diagonal positions.
\section{Preliminaries}
\label{section_preliminaries}

We consider a linear programming (LP) problem in the following standard form:
\begin{equation}
  \label{problem_main}
  \min_{x \in \real^n} c^{\top} x, \quad \text { subject to } A x = b, \quad x \ge 0,
\end{equation}
where $A \in \real^{m \times n}$, $b \in \real^m$, and $c \in \real^n$.
The corresponding dual problem for (\ref{problem_main}) is
\begin{equation}
  \label{problem_dual}
  \max_{\lambda \in \real^m, \ s \in \real^n} b^{\top} \lambda, \quad \text { subject to } A^{\top} \lambda+s=c, \quad s \ge 0,
\end{equation}
where $\lambda$ is the dual variable vector, and $s$ is the dual slack vector.
Let $\SolSet$ denote the set of the optimal solutions
of (\ref{problem_main}) and (\ref{problem_dual}).
It is well known that
$(x^*, \lambda^*, s^*) \in \SolSet$
if $(x^*, \lambda^*, s^*)$ satisfies the following KKT conditions:
\begin{subequations}
  \label{KKT_conditions}
  \begin{align}
    A x^* & = b \\
    A^\top \lambda^* + s^* & = c \\
    (x^*,s^*) & \ge 0 \\
    x_i^* s_i^* & = 0, \quad i = 1,\ldots,n.
  \end{align}
\end{subequations}

We denote the residuals of the equality constraints in \eqref{problem_main} and \eqref{problem_dual} by
\begin{subequations}
  \label{residuals_constraints}
  \begin{align}
    r_b(x) & = A x - b \label{residual_main} \\
    r_c(\lambda, s) & = A^\top \lambda + s - c, \label{residual_dual}
  \end{align}
\end{subequations}
and the duality measure by
$$
  \mu = \frac{x^T s}{n}.
$$

In this paper,
we make the following assumption for \eqref{problem_main} and \eqref{problem_dual}.
This assumption is a mild one and is used in many papers---for example, see \cite{wright1997primal,yang2020arc}.
\begin{assumption}
  \label{assumption_for_optimal_solution}
	\begin{itemize}
    \item[]
		\item[(A1)] Each of the primal problem \eqref{problem_main} and
		the dual problem \eqref{problem_dual} has interior feasible points.
		\item[(A2)] The row vectors in $A$ are linear independent.
	\end{itemize}
\end{assumption}
Assumption~\ref{assumption_for_optimal_solution} guarantees the existence of
optimal solutions and the boundedness of $\SolSet$---see \cite{yang2020arc}.

\subsection{Arc-search interior-point method} \label{existing-method}

Interior-point methods are iterative methods, and we denote the $k$th iteration by
$(x^k, \lambda^k, s^k)\in \real^n \times \real^m \times \real^n$; in particular,
the initial point is
$(x^0, \lambda^0, s^0)$.

Given a strictly positive iteration
$(x^k, \lambda^k, s^k)$ such that $(x^k, s^k)>0$,
an infeasible predictor-corrector method~\cite{yang2020arc} traces a smooth curve called an approximate center path:
\begin{equation}
  \label{def_ellipsoid}
	C = \left\{(x(t), \lambda(t), s(t)) \mid t \in (0,1] \right\},
\end{equation}
where $(x(t), \lambda(t), s(t))$ is the unique solution of the following system
\begin{subequations}
  \label{curve_to_optimal_solution}
  \begin{align}
    &A x(t) - b = t \ r_b(x^k) \\
    &A^\top \lambda(t) + s(t) - c = t \ r_c(\lambda^k, s^k) \\
    &x(t) \circ s(t) = t \ (x^k \circ s^k) \\
    &(x(t), s(t))>0.
  \end{align}
\end{subequations}
As $t$ approaches 0, $(x(t), \lambda(t), s(t))$ converges to an optimal solution
in $\SolSet$.
Though the existence of the curve $C$ is guaranteed by the implicit function theorem,
it does not admit an analytical form.

The key idea in the arc-search type interior-point algorithms~\cite{Yang2017} is to approximate the curve $C$
with an ellipsoidal arc in the $2 n+m$ dimensional space.
An ellipsoidal approximation of $(x(t), \lambda(t), s(t))$ at
$(x^k, \lambda^k, s^k)$ for an angle $\alpha \in (0, \pi)$ is explicitly given by
$(x(\alpha), \lambda(\alpha), s(\alpha))$ with
\begin{subequations}
  \label{def_x_lambda_s_alpha}
  \begin{align}
    x(\alpha) = x - \dot{x}\sin(\alpha)+\ddot{x}(1-\cos(\alpha)), \\
    \lambda(\alpha) = \lambda-\dot{\lambda}\sin(\alpha)+\ddot{\lambda}(1-\cos(\alpha)), \\
    s(\alpha) = s - \dot{s}\sin(\alpha)+\ddot{s}(1-\cos(\alpha)). \label{sAlpha}
  \end{align}
\end{subequations}
Here,
the first derivative $(\dot{x}, \dot{\lambda}, \dot{s})$ and the
second derivative $(\ddot{x}, \ddot{\lambda}, \ddot{s})$ are the solutions
of the following systems:
\begin{align}
  \label{first_derivative_original}
  \begin{bmatrix}
    A & 0 & 0 \\
    0 & A^\top & I \\
    S^k & 0 & X^k
  \end{bmatrix} \left[\begin{array}{l}
    \dot{x} \\
    \dot{\lambda} \\
    \dot{s}
  \end{array}\right]
  &=
  \left[\begin{array}{c}
    r_b(x^k) \\
    r_c(\lambda^k, s^k) \\
    x^k \circ s^k
  \end{array}\right] \\
  \label{second_derivative_original}
  \begin{bmatrix}
    A & 0 & 0 \\
    0 & A^\top & I \\
    S^k & 0 & X^k
  \end{bmatrix} \left[\begin{array}{l}
    \ddot{x} \\
    \ddot{\lambda} \\
    \ddot{s}
  \end{array}\right]
  &=
  \left[\begin{array}{c}
    0 \\
    0 \\
    - 2 \dot{x} \circ \dot{s}
  \end{array}\right].
\end{align}

In the remainder of this section,
we introduce definitions that are necessary to discuss the convergence.
In the same way as Yang~\cite{yang2020arc},
we introduce the neighborhood of the central path by
\begin{equation}
  \label{def_neighborhood}
  \Neighborhood(\theta) := \left\{
    (x, \lambda, s) \mid (x, s)>0, \quad
    \norm{x \circ s - \mu \mathrm{e}} \le \theta \mu
  \right\}.
\end{equation}
By generating a  sequence in this neighborhood $\Neighborhood(\theta)$,
Yang~\cite{yang2020arc} proved that the algorithm converges in polynomial time.
Similarly to Miao \cite{miao1996two} and Kojima \cite{kojima1996basic},
we also choose a sufficiently large parameter $\rho \ge 1$ and  an initial point $\left(x^{0}, \lambda^{0}, s^{0}\right)$ such that
\begin{equation}
  \label{condition_initial_and_optim_point}
  \left(x^0, \lambda^0, s^0\right) \in \Neighborhood(\theta), \quad
  x^{*} \le \rho x^{0}, \quad s^{*} \le \rho s^{0}
\end{equation}
for some $\left(x^{*}, \lambda^{*}, s^{*}\right) \in \SolSet$.
Let $\omega^f$ and $\omega^o$ be the distances of feasibility and optimality from the initial point as follows:
\begin{align}
  \omega^f & = \min _{x, \lambda, s} \left\{ \begin{aligned}
    \max \left\{ \norm{ \left( X^0 \right)^{-1}\left(x-x^0\right)}_{\infty},
    \norm{\left( S^0 \right)^{-1}\left(s-s^0\right)}_{\infty}\right\} \\
    \mid Ax = b, A^\top \lambda+s=c, (x,s) \ge 0
  \end{aligned} \right\} \label{definition_omega_f} \\
  \omega^o & = \min _{x^{*}, \lambda^{*}, s^{*}} \left\{
    \max \left\{\frac{(x^*)^\top s^0}{(x^0)^\top s^0}, \frac{(s^*)^\top x^0}{(x^0)^\top s^0}, 1\right\}
    \mid \left(x^*, \lambda^*, s^*\right) \in \SolSet
  \right\} \label{definition_omega_o}.
\end{align}

From \eqref{definition_omega_f},
there exists a feasible solution $(\bar{x}, \bar{\lambda}, \bar{s})$ satisfying
\begin{equation}
  \label{upper_feasible_solution}
  \abs{x_i^0-\bar{x}_i} \le \omega^f x_i^0, \quad \abs{s_i^0-\bar{s}_i} \le \omega^f s_i^0.
\end{equation}
In addition, from \eqref{definition_omega_o},
there is an optimal solution $\left(x^{*}, \lambda^{*}, s^{*}\right) \in \SolSet$ satisfying
\begin{equation}
  \label{upper_x_s_solution_product}
  (x^0)^\top s^*, (x^*)^\top s^0 \le \omega^o (x^0)^\top s^0.
\end{equation}

We will use these definitions to discuss the convergence of the proposed method.
\section{The proposed method}
\label{section_proposed_method}
We introduce Nesterov's restarting strategy~\cite{nesterov1983} for the arc-search interior-point method.
Nesterov's restarting strategy has been described in the literature to accelerate convergence in gradient descent algorithms,
and Morshed and Noor-E-Alam~\cite{Morshed2020}
applied Nesterov's restarting strategy to an affine scaling method.

To combine Nesterov's restarting strategy into the arc-search interior-point method,
we employ the momentum term
\begin{equation}
  \label{def_delta}
  \delta(x^k) = x^k - x^{k-1}
\end{equation}
with a strictly positive point $x^0 > 0$,
and we construct $z_k$ by
\begin{equation}
  \label{def_restarted_point}
  z^k =
    \begin{cases}
      x^k + \beta_k \delta(x^k) & k > 0 \ \textrm{and} \ (x^k + \beta_k \delta(x^k), \lambda^k, s^k) \in \Neighborhood(\theta) \\
      x^k & \textrm{otherwise},
    \end{cases}
  \end{equation}
where $\beta_k \ge 0$ is a weight of the momentum term which will be calculated by
\begin{equation}
  \label{def_beta_for_proof}
  \beta_k = \min \left\{
    \frac{\beta}{\norm{ X_k^{-1}\delta(x^k) }_\infty},
    \min \left\{\begin{aligned}
      & \abs{ \frac{r_b(x^k)_j}{r_b(x^k)_j - r_b(x^{k-1})_j}} \\
      & \text{for} \ j \in \left\{r_b(x^k)_j - r_b(x^{k-1})_j \ne 0\right\}
    \end{aligned}
    \right\}
  \right\}.
\end{equation}
Here, $\beta \in (0, 1)$ is a parameter that is given before starting algorithms.
In \eqref{def_restarted_point},
$z^k$ with the momentum $\delta(x^k)$ is adopted if it is in the neighborhood $\Neighborhood(\theta)$.
We use this $z^k$ instead of  $x^k$ to compute the arc in each iteration.
The first and the second derivatives at $(z^k, \lambda^k, s^k)$ with respect to $t$ are given by
\begin{align}
  \label{first_derivative}
  \begin{bmatrix}
    A & 0 & 0 \\
    0 & A^\top & I \\
    S^k & 0 & Z^k
  \end{bmatrix} \left[\begin{array}{c}
    \dot{z} \\
    \dot{\lambda} \\
    \dot{s}
  \end{array}\right]
  & =
  \left[\begin{array}{c}
    r_b(z^k) \\
    r_c(\lambda^k, s^k) \\
    z^k \circ s^k
  \end{array}\right] \\
  \label{second_derivative}
  \begin{bmatrix}
    A & 0 & 0 \\
    0 & A^\top & I \\
    S^k & 0 & Z^k
  \end{bmatrix} \left[\begin{array}{c}
    \ddot{z} \\
    \ddot{\lambda} \\
    \ddot{s}
  \end{array}\right]
  & =
  \left[\begin{array}{c}
    0 \\
    0 \\
    - 2 \dot{z} \circ \dot{s},
  \end{array}\right],
\end{align}
and we denote the duality measure at $(z^k, \lambda^k, s^k)$ by
\begin{equation}
  \label{def_duality_measure_with_restarting_point}
  \mu_k^z := \frac{(z^k)^\top s^k}{n}.
\end{equation}

A framework of the proposed method is given as~Algorithm \ref{algo_arc_restarting_for_proof}.
\begin{algorithm}[H]
  \caption{Arc-search interior-point method with Nesterov's restarting strategy}
  \label{algo_arc_restarting_for_proof}
  \begin{algorithmic}[1]
    \renewcommand{\algorithmicrequire}{\textbf{Input:}}
    \renewcommand{\algorithmicensure}{\textbf{Output:}}
    \Require a neighborhood parameter $\theta \in (0, 1/(2 + \sqrt{2}))$,
      an initial point $(x^0, \lambda^0, s^0)  \in \Neighborhood(\theta)$,
      restarting parameter $\beta \in (0, 1)$,
      and a stopping threshold $\epsilon \in (0, 1)$.
    \State If a stopping criterion is satisfied,
      then output $(x^k, \lambda^k, s^k)$ as an optimal solution $(x^*, \lambda^*, s^*)$ and stop.
    \State Set $\beta_k$ by \eqref{def_beta_for_proof}.
    \State  Set $z^k$ by \eqref{def_restarted_point} and $\mu_k^z$ by \eqref{def_duality_measure_with_restarting_point}.
    \State Compute $(\dot{z}, \dot{\lambda}, \dot{s})$ and
      $(\ddot{z}, \ddot{\lambda}, \ddot{s})$ using \eqref{first_derivative} and \eqref{second_derivative}, respectively.
    \State Find the largest $\alpha_k \in (0, \pi/2]$ such that for all $\alpha \in (0, \alpha_k]$, the following inequalities hold:
      \begin{subequations}
        \label{condition_step_size_for_proof}
        \begin{align}
          x(\alpha) = z^{k}-\dot{z} \sin (\alpha)+\ddot{z}(1-\cos (\alpha)) & > 0 \\
          s(\alpha) = s^{k}-\dot{s} \sin (\alpha)+\ddot{s}(1-\cos (\alpha)) & > 0 \\
          \norm{(x(\alpha) \circ s(\alpha))-(1-\sin (\alpha)) \mu_k^z e } & \le 2 \theta(1-\sin (\alpha)) \mu_k^z
          \label{condition_step_size_neighborhood}
        \end{align}
      \end{subequations}
    \State Set
      \begin{subequations}
        \label{alpha_update_formula}
        \begin{align}
          x(\alpha_k) & = z^k-\dot{z} \sin \left(\alpha_k\right)+\ddot{z}\left(1-\cos \left(\alpha_k\right)\right) \label{x_alpha_update_formula} \\
          \lambda(\alpha_k) & = \lambda^{k}-\dot{\lambda} \sin \left(\alpha_k\right)+\ddot{\lambda}\left(1-\cos \left(\alpha_k\right)\right) \\
          s(\alpha_k) & = s^k-\dot{s} \sin \left(\alpha_k\right)+\ddot{s}\left(1-\cos \left(\alpha_k\right)\right)
        \end{align}
      \end{subequations}
    \State Calculate $(\Delta x, \Delta \lambda, \Delta s)$ by solving
      \begin{equation}
        \label{corrector_calculation}
        \left[\begin{array}{ccc}
          A & 0 & 0 \\
          0 & A^\top & I \\
          S\left(\alpha_k\right) & 0 & X\left(\alpha_k\right)
        \end{array}\right]\left[\begin{array}{c}
          \Delta x \\
          \Delta \lambda \\
          \Delta s
        \end{array}\right]=\left[\begin{array}{c}
          0 \\
          0 \\
          \left(1-\sin \left(\alpha_k\right)\right) \mu_k e-x\left(\alpha_k\right) \circ s\left(\alpha_k\right)
        \end{array}\right].
      \end{equation}
    \State Update
      \begin{equation}
        \label{corrector_update_formula}
        \left(x^{k+1}, \lambda^{k+1}, s^{k+1}\right)=\left(x(\alpha_k), y(\alpha_k), s(\alpha_k)\right)+(\Delta x, \Delta \lambda, \Delta s)
      \end{equation}
      and $\mu_{k+1} = (x^{k+1})^\top s^{k+1} / n$.
    \State Set $k \leftarrow k+1$. Go back to Step 1.
  \end{algorithmic}
\end{algorithm}
\section{Theoretical proof}
\label{section_theoretical_proof}
We discuss theoretical aspects of Algorithm~\ref{algo_arc_restarting_for_proof},
in particular,
we focus on the convergence of the generated sequence and the number of iterations.
In this section,
let $\{(x^k, \lambda^k, s^k)\}$ be the sequence generated by Algorithm~\ref{algo_arc_restarting_for_proof}.

To compute the first and the second derivatives at $(z^k, \lambda^k, s^k)$
uniquely by \eqref{first_derivative} and \eqref{second_derivative},
the following matrix must be nonsingular for all $k$:
\begin{equation}
  \label{matrix_need_nonsingularity}
  \begin{bmatrix}
    A & 0 & 0 \\
    0 & A^\top & I \\
    S^k & 0 & Z^k
  \end{bmatrix}.
\end{equation}
Since $A$ is full rank matrix from Assumption~\ref{assumption_for_optimal_solution}
and $S^k$ is a positive diagonal matrix due to  Proposition~\ref{proposition_s_bounded_below_away_from_zero} below,
the nonsingularity of \eqref{matrix_need_nonsingularity} is ensured.
\begin{proposition}
  \label{proposition_s_bounded_below_away_from_zero}
  There exists $L_0 > 0$ such that $s_i^k \ge L_0 \mu_k$ for each $i=1,\dots,n$.
\end{proposition}
The proof of Proposition~\ref{proposition_s_bounded_below_away_from_zero} will be given later.
The following lemma indicates that $z^k$ is a positive vector.
\begin{lemma}
  Suppose  $x^k$ and $z^k$ are obtained at the $k$th iteration of Algorithm~\ref{algo_arc_restarting_for_proof}.
  Then, for all $i \in \{1,2,\dots,n\}$,
  \begin{equation}
    \label{upper_and_lower_z_by_x}
    (1-\beta) x^k_i \le z^k_i \le (1+\beta) x^k_i.
  \end{equation}
\end{lemma}

\begin{proof}
  If $z^k = x^k$, it is clear.
  When $z^k = x^k + \beta_k \delta(x^k)$, we have
  $$
    \frac{\abs{\delta(x^k)_i}}{\norm{ X_k^{-1}\delta(x^k) }_\infty}
    = \frac{\abs{x^k_i - x^{k-1}_i}}{\max_j\abs{ \frac{x_j^k - x_j^{k-1}}{x_j^k}}}
    \le \frac{\abs{x^k_i - x^{k-1}_i}}{\frac{\abs{x_i^k - x_i^{k-1} }}{\abs{x_i^k}}}
    = \abs{x^k_i} = x_i^k,
  $$
  for all $i$, where the last equality holds from $x^k > 0$.
  The definition of $\beta_k$ \eqref{def_beta_for_proof} indicates
  $$
    \abs{\beta_k \delta(x^k)_i} \le \beta \frac{\abs{\delta(x^k)_i}}{\norm{ X_k^{-1}\delta(x^k) }_\infty} \le \beta x_i^k,
  $$
  so we can get \eqref{upper_and_lower_z_by_x}
  since $\beta \in (0,1)$.
\end{proof}

From \eqref{upper_and_lower_z_by_x}, we obtain $z^k \ge (1-\beta) x^k > 0$.
Furthermore,  from $s^k > 0$, we can also derive the following corollary.

\begin{corollary}
  \label{corollary_z_s_boundedness}
  Suppose  $x^k$, $s^k$ and $z^k$ are obtained at the $k$th iteration of Algorithm~\ref{algo_arc_restarting_for_proof}.
  Then,
  \begin{equation}
    \label{upper_and_lower_z_s_by_x_s}
    (1-\beta) x^k_i s^k_i \le z^k_i s^k_i \le (1+\beta) x^k_i s^k_i.
  \end{equation}
\end{corollary}

\subsection{Convergence of Algorithm~\ref{algo_arc_restarting_for_proof}}
To discuss the convergence of the generated sequence by Algorithm~\ref{algo_arc_restarting_for_proof}, we first evaluate
the residuals of the equality constraints.
Some proofs in this section can be done in similar ways to the reference~\cite{yang2020arc},
but will be described in this paper for the sake of completeness,
since we are using $z^k$ instead of $x^k$.

To prove the monotonic decrement of the residuals in the constraints \eqref{residuals_constraints},
we first show that the signs in the primal constraint remain same.

\begin{lemma}
  \label{lemma_non_change_sign_of_main_constraints}
  Let $x^k$ and $z^k$ be generated by Algorithm \ref{algo_arc_restarting_for_proof}. Then,
  \begin{enumerate}[(i)]
    \item $r_b(x^k)_j \ge r_b(z^k)_j \ge 0$ if $r_b(x^0)_j \ge 0$,
    \item $r_b(x^k)_j \le r_b(z^k)_j \le 0$ if $r_b(x^0)_j \le 0$.
  \end{enumerate}
\end{lemma}

\begin{proof}
  From \eqref{residual_main}, \eqref{corrector_update_formula}, \eqref{x_alpha_update_formula},
  \eqref{second_derivative}, \eqref{corrector_calculation} and \eqref{first_derivative},
  it holds that
  \begin{align*}
    r_b(x^{k+1}) - r_b(z^k) & = A\left(x^{k+1} - z^k\right) = A\left(x(\alpha_k) + \Delta x - z^k\right) \\
    & = A \left(z^k-\dot{z} \sin \left(\alpha_k\right)+\ddot{z}\left(1-\cos \left(\alpha_k\right)\right) + \Delta x - z^k\right) \\
    & = - A \dot{z} \sin(\alpha_k) \\
    & = - r_b(z^k) \sin(\alpha_k),
  \end{align*}
  therefore, we can get
  \begin{equation}
    r_b(x^{k+1}) = r_b(z^k) \left( 1 - \sin(\alpha_k) \right).
    \label{eq_main_residual_decreasing_by_z}
  \end{equation}

  We give the proof by induction on the iteration number $k$.
  If $k=0$ or $z^k = x^k$ in \eqref{def_restarted_point} is satisfied,
  thus both (i) and (ii) hold.
  We assume $\abs{r_b(z^k)_j} \le \abs{r_b(x^k)_j}$ with $k=t$,
  and discuss the case of $k = t+1$ and $z^k = x^k + \beta_k \delta(x^k)$.
  From \eqref{residual_main} and \eqref{def_delta}, we derive
  \begin{align}
    & \ r_b(z^{t+1}) - r_b(x^{t+1}) = A z^{t+1} - A x^{t+1}
     = \beta_{t+1} A \delta(x^{t+1}) \notag \\
    = & \ \beta_{t+1} A (x^{t+1} - x^t)
    = \beta_{t+1} (r_b(x^{t+1}) - r_b(x^t)). \label{diff_residual_main}
  \end{align}

  We consider the case of (i).
  From \eqref{eq_main_residual_decreasing_by_z} and $r_b(x^t)_j \ge r_b(z^t)_j$, we know
  \begin{equation*}
    r_b(x^{t+1})_j - r_b(x^t)_j = r_b(z^t)_j \left( 1 - \sin(\alpha_t) \right) - r_b(x^t)_j
    \le - r_b(x^t)_j \sin(\alpha_t) \le 0,
  \end{equation*}
  thus, we get $r_b(x^{t+1})_j \ge r_b(z^{t+1})_j$ from \eqref{diff_residual_main}.
  In addition,
  we can get $r_b(x^{t+1})_j \ge 0$ from \eqref{eq_main_residual_decreasing_by_z} and $r_b(z^t)_j \ge 0$,
  and
  $$
    \beta_{t+1} \le \frac{\abs{r_b(x^{t+1})_j}}{\abs{r_b(x^{t+1})_j - r_b(x^t)_j}}
  $$
  from \eqref{def_beta_for_proof}, so we obtain
  \begin{align*}
    r_b(z^{t+1})_j & = r_b(x^{t+1})_j + \beta_{t+1} (r_b(x^{t+1})_j - r_b(x^t)_j) \\
    & \ge r_b(x^{t+1})_j + \frac{\abs{r_b(x^{t+1})_j}}{\abs{r_b(x^{t+1})_j - r_b(x^t)_j}} (r_b(x^{t+1})_j - r_b(x^t)_j) \\
    & = r_b(x^{t+1})_j + \frac{r_b(x^{t+1})_j}{r_b(x^t)_j - r_b(x^{t+1})_j} (r_b(x^{t+1})_j - r_b(x^t)_j) \\
    & = 0,
  \end{align*}
  therefore, we can get $r_b(x^{t+1})_j \ge r_b(z^{t+1})_j \ge 0$ and this shows the case of (i).
  The proof for (ii) can be given similarly.
\end{proof}

Next, we prove the convergence of the residuals of constraints.

\begin{theorem}
  \label{theorem_residual_decreasing}
  The residuals \eqref{residuals_constraints} shrink at a rate of at least $(1-\sin(\alpha_k))$ in each iteration,
  more precisely,
  \begin{align*}
    \abs{r_b(x^{k+1})_j} & \le \abs{ r_b(x^k)_j } \left( 1 - \sin(\alpha_k) \right) \quad \forall j \in \{1, \ldots, m\}\\
    r_c(\lambda^{k+1}, s^{k+1}) & = r_c(\lambda^k, s^k) \left( 1 - \sin(\alpha_k) \right).
  \end{align*}
\end{theorem}

\begin{proof}
  From Lemma~\ref{lemma_non_change_sign_of_main_constraints},
  $\abs{ r_b(z^t)_j } \le \abs{ r_b(x^t)_j }$ holds regardless of the sign of  $r_b(x^0)_j$.
  Hence, from \eqref{eq_main_residual_decreasing_by_z}, we obtain
  $\abs{ r_b(x^{t+1})_j} = \abs{r_b(z^t)_j }\left(1 - \sin(\alpha_t) \right) \le \abs{ r_b(x^t)_j }\left( 1 - \sin(\alpha_t)\right)$
  for each $j$.

  Lastly, we prove the second part.
  From \eqref{residual_dual}, \eqref{corrector_update_formula}, \eqref{alpha_update_formula},
  \eqref{second_derivative}, \eqref{corrector_calculation} and \eqref{first_derivative}, we know
  \begin{align*}
    r_c(\lambda^{k+1}, s^{k+1}) - r_c(\lambda^k, s^k) = & A^\top \left(\lambda(\alpha_k) + \Delta \lambda - \lambda^k\right) + \left(s(\alpha_k) + \Delta s - s^k\right) \\
    = & -\left(A^\top \dot{\lambda} + \dot{s}\right)\sin(\alpha_k) + \left(A^\top \ddot{\lambda} + \ddot{s}\right)(1-\cos(\alpha_k)) \\
    & + \left(A^\top \Delta \lambda + \Delta s\right) \\
    = & -r_c(\lambda^k, s^k)\sin(\alpha_k).
  \end{align*}
\end{proof}

Owing to the momentum term,
the residual of every primal constraint can shrink at the ratio of $(1-\sin(\alpha_k))$ or faster.
To evaluate the primal residual $r_b(z^k)$ more precisely,
we introduce the following lemma.
\begin{lemma}
  For each $k$, there exists $\tau_k$ such that $0 \le \tau_k \le 1$ and
  \begin{equation}
    r_b(x^{k+1}) = (1 - \sin(\alpha_k)) \tau_k r_b(x^k).
    \label{main_residual_norm_decreasing_by_iteration}
  \end{equation}
\end{lemma}

\begin{proof}
  We give a proof by the induction on the iteration number $k$.
  When $k=0$, $r_b(x^1) = (1 - \sin(\alpha_0)) r_b(x^0)$,
  then \eqref{main_residual_norm_decreasing_by_iteration} is satisfied with $\tau_0 = 1$.
  For a positive integer $t$,
  we assume that \eqref{main_residual_norm_decreasing_by_iteration} holds for any $k$ such that $k \le t$.
  The sign of each component of $r_b(x^t)$ does not change
  due to Lemma~\ref{lemma_non_change_sign_of_main_constraints},
  thus we know  $\tau_k \ge 0$ for $k \le t$.

  If $1 - \sin(\alpha_t) = 0$ or $\tau_t = 0$ hold,
   Theorem~\ref{theorem_residual_decreasing} implies
  $r_b(x^{t^\prime}) = 0$ for all $t^\prime \ge t$.
  Then,
  we can choose arbitrary $\tau_{t^\prime}$ from the range $0 \le \tau_{t^\prime} \le 1$, and we can satisfy \eqref{main_residual_norm_decreasing_by_iteration}.
  From here, therefore, we focus on the case at $k=t+1$ when $1 - \sin(\alpha_{k^\prime}) > 0$ and $\tau_{k^\prime} > 0$ hold for all $k^\prime \le t$.

  From \eqref{eq_main_residual_decreasing_by_z}, we have
  $r_b(x^{t+2}) = (1-\sin(\alpha_{t+1})) r_b(z^{t+1})$.
  Due to \eqref{def_restarted_point} and \eqref{eq_main_residual_decreasing_by_z},
  \begin{align*}
    r_b(z^{t+1}) & = r_b(x^{t+1}) + \beta_{t+1} (r_b(x^{t+1}) - r_b(x^t))\\
    & = r_b(x^{t+1}) + \beta_{t+1} \left(1 - \frac{1}{(1 - \sin(\alpha_t)) \tau_t}\right) r_b(x^{t+1}).
  \end{align*}
  From $0 < 1 - \sin(\alpha_t) \le 1$, $0 < \tau_t \le 1$,
  and \eqref{def_beta_for_proof}, i.e., $\beta_{t+1} \ge 0$, we can get
  \begin{equation*}
    \beta_{t+1} \left(1 - \frac{1}{(1 - \sin(\alpha_t)) \tau_t}\right) \le \beta_{t+1}(1 - 1) = 0.
  \end{equation*}
  Therefore,
  $\tau_{t+1} = \left(1 + \beta_{t+1} - \frac{\beta_{t+1}}{(1 - \sin(\alpha_t)) \tau_t}\right)$ satisfies $\tau_{t+1} \le 1$ and
  \begin{align*}
    r_b(x^{t+2}) & = (1 - \sin(\alpha_{t+1}))r_b(z^{t+1}) \\
    & = (1 - \sin(\alpha_{t+1})) \tau_{t+1} r_b(x^{t+1}).
  \end{align*}
  On the other hand, from Lemma~\ref{lemma_non_change_sign_of_main_constraints}, we know
  $\tau_{t+1} \ge 0$.
  Hence, there exists $\tau_k$ such that $0 \le \tau_k \le 1$ and \eqref{main_residual_norm_decreasing_by_iteration} when $k = t+1$.
\end{proof}

The next lemma shows the reduction speed of the duality measure.

\begin{lemma}
  \label{lemma_mu_decreasing}

  Let the sequence $\{(x^k, \lambda^k, s^k)\}$ be generated by Algorithm~\ref{algo_arc_restarting_for_proof}.
  Then, it holds that
  $$
    \mu_{k+1}=\mu_k \left(1-\sin \left( \alpha_k \right)\right).
  $$
\end{lemma}
Lemma~\ref{lemma_mu_decreasing} can be proved in the same way as Yang~\cite[Lemma~3.1]{yang2018arc},
because this lemma does not depend on using $z^k$,
therefore we omit the proof.

For the following discussions, let
\begin{subequations}
  \label{def_mus}
  \begin{align}
    \nu_k = \prod_{i=0}^{k-1} (1 - \sin(\alpha_i)), \label{def_nu}\\
    \nu_k^\tau = \prod_{i=0}^{k-1} (1 - \sin(\alpha_i))\tau_i \label{def_nu_tau}.
  \end{align}
\end{subequations}
From \eqref{main_residual_norm_decreasing_by_iteration}, Theorem~\ref{theorem_residual_decreasing} and Lemma~\ref{lemma_mu_decreasing},
we can get
\begin{subequations}
  \begin{align}
    r_b(z^k) & = \tau_k \nu_k^\tau r_b(x^0), \label{main_residual_norm_decreasing} \\
    r_c(\lambda^k, s^k) & = \nu_k r_c(\lambda^0, s^0), \label{dual_residual_decreasing} \\
    \mu_k & = \nu_k \mu_0. \label{mu_decreasing}
  \end{align}
\end{subequations}

Next,
Lemmas~\ref{lemma_alpha_satisfying_x_s_positive} and \ref{lemma_next_iteration_in_neighborhood} show that Algorithm~\ref{algo_arc_restarting_for_proof} is well-defined,
i.e., it continues the iterations with positive step sizes.

\begin{lemma}
  \label{lemma_alpha_satisfying_x_s_positive}
  Assume $(z^k, \lambda^k, s^k) \in \Neighborhood(\theta)$.
  Then, there is a $\bar{\alpha}_k \in (0, \pi / 2)$
  such that $(x(\alpha), s(\alpha)) > 0$ and \eqref{condition_step_size_neighborhood} hold for any $\alpha \in (0, \bar{\alpha}_k]$.
\end{lemma}
\begin{proof}
  We can find
  $s^k \circ \dot{z} + z^k \circ \dot{s} = z^k \circ s^k$ and
  $s^k \circ \ddot{z} + z^k \circ \ddot{s} = -2 \dot{z} \circ \dot{s}$
  in the last rows of (\ref{first_derivative}) and (\ref{second_derivative}), respectively.
  Using $\sin^2(\alpha) - 2 (1 - \cos(\alpha)) = - (1 - \cos(\alpha))^2$, we have
  \begin{equation*}
    \begin{aligned}
      x(\alpha) \circ s(\alpha) = & z^k \circ s^k - (s^k \circ \dot{z} + z^k \circ \dot{s})\sin(\alpha) + (s^k \circ \ddot{z} + z^k \circ \ddot{s})(1 - \cos(\alpha)) \\
      & + \dot{z} \circ \dot{s} \sin^2(\alpha) - (\dot{z} \circ \ddot{s} + \ddot{z} \circ \dot{s}) \sin(\alpha) (1 - \cos(\alpha)) \\
      & + \ddot{z} \circ \ddot{s}(1 - \cos(\alpha))^2 \\
      = & z^k \circ s^k (1 - \sin(\alpha)) + \dot{z} \circ \dot{s} (\sin^2(\alpha) - 2 (1 - \cos(\alpha))) \\
      & - (\dot{z} \circ \ddot{s} + \ddot{z} \circ \dot{s}) \sin(\alpha) (1 - \cos(\alpha)) + \ddot{z} \circ \ddot{s}(1 - \cos(\alpha))^2 \\
      = & z^k \circ s^k (1 - \sin(\alpha)) + (\ddot{z} \circ \ddot{s} - \dot{z} \circ \dot{s})(1 - \cos(\alpha))^2 \\
      & - (\ddot{z} \circ \dot{s} + \dot{z} \circ \ddot{s})\sin(\alpha)(1 - \cos(\alpha)).
    \end{aligned}
  \end{equation*}
  Therefore, it holds that
  \begin{align*}
    & \norm{(x(\alpha) \circ s(\alpha)) - (1-\sin (\alpha)) \mu_k^z e } \\
    = & \|
      z^k \circ s^k (1 - \sin(\alpha)) + (\ddot{z} \circ \ddot{s} - \dot{z} \circ \dot{s})(1 - \cos(\alpha))^2 \\
    & - (\ddot{z} \circ \dot{s} + \dot{z} \circ \ddot{s})\sin(\alpha)(1 - \cos(\alpha))
      - (1-\sin (\alpha)) \mu_k^z e
    \| \\
    = & \|
      (z^k \circ s^k - \mu_k^z e) (1 - \sin(\alpha)) + (\ddot{z} \circ \ddot{s} - \dot{z} \circ \dot{s})(1 - \cos(\alpha))^2 \\
    & - (\ddot{z} \circ \dot{s} + \dot{z} \circ \ddot{s})\sin(\alpha)(1 - \cos(\alpha))
    \| \\
    \le & \theta \mu_k^z (1 - \sin(\alpha))
      + (\norm{\ddot{z} \circ \ddot{s}} + \norm{\dot{z} \circ \dot{s}})\sin^4(\alpha)
      + (\norm{\ddot{z} \circ \dot{s}} + \norm{\dot{z} \circ \ddot{s}})\sin^3(\alpha),
  \end{align*}
  where the last inequality was derived by
  $(z^k, \lambda^k, s^k) \in \Neighborhood(\theta)$ from \eqref{def_restarted_point}
 and  $0 \le 1 - \cos(\alpha) \le \sin^2(\alpha)$ for $\alpha \in (0, \pi/2)$.
  Clearly, if
  \begin{equation}
    \label{q_negativity}
    q(\alpha) := - \theta \mu_k^z (1 - \sin(\alpha))
      + (\norm{\ddot{z} \circ \ddot{s}} + \norm{\dot{z} \circ \dot{s}})\sin^4(\alpha)
      + (\norm{\ddot{z} \circ \dot{s}} + \norm{\dot{z} \circ \ddot{s}})\sin^3(\alpha)
    \le 0
  \end{equation}
  then \eqref{condition_step_size_neighborhood} holds.
  In fact,
  due to $q(0) = -\mu_k^z \theta < 0$,
  $q(\pi/2) > 0$ and continuity,
  there exists a $\bar{\alpha}_k \in (0, \pi/2]$ such that \eqref{q_negativity} is satisfied at any $\alpha \in (0, \bar{\alpha}_k]$,
  This guarantees (\ref{condition_step_size_neighborhood}), thus we can derive that
  $$
    x_i(\alpha) s_i(\alpha) \ge (1 - 2\theta)(1 - \sin(\alpha)) \mu_k^z > 0,
    \quad \forall \theta \in [0, 0.5), \quad \forall \alpha \in (0, \bar{\alpha}_k].
  $$
  Since $x(\alpha)$ and $s(\alpha)$ are continuous with respect to $\alpha$ and $(x(0), s(0)) = (x^k, s^k) > 0$,
  we can get $(x(\alpha), s(\alpha)) > 0$ for any $\alpha \in (0, \bar{\alpha}_k]$.
\end{proof}

\begin{lemma}
  \label{lemma_next_iteration_in_neighborhood}
  Assume $(z^k, \lambda^k, s^k) \in \Neighborhood(\theta)$.
  If $\theta \le 1 / (2+\sqrt{2})$,
  $\left(x^{k+1}, \lambda^{k+1}, s^{k+1}\right) \in \Neighborhood(\theta)$
  and $x^{k+!}, s^{k+1} > 0$.
\end{lemma}
We can prove this lemma in the same way as Yang~\cite[Theorem~3.1]{yang2018arc}
by applying Lemmas~\ref{lemma_mu_decreasing} and \ref{lemma_alpha_satisfying_x_s_positive}.

Next, we prepare several lemmas to prove in Lemma~\ref{lemma_lower_bound_sin_alpha}
that there exists a strictly positive constant $\bar{\alpha} > 0$
such that $\alpha_k > \bar{\alpha}$ through all iterations until the algorithm termination.
In Lemmas~\ref{lemma_upper_norm_D_dot} and \ref{lemma_upper_derivative_norms} below,
we will derive upper bounds for $\dot{z}, \dot{s}, \ddot{z}, \ddot{s}$.
For these lemmas, we first prove the upper bound of $\tau_k \nu_k^\tau (s^k)^\top x^0 + \nu_k (z^k)^\top s^0$ by letting
\begin{subequations}
  \label{def_hat_x_lambda_s}
  \begin{align}
    \hat{x} & =\tau_k \nu_k^\tau {x}^0+(1 - \tau_k \nu_k^\tau) x^* \label{def_hat_x} \\
    (\hat{\lambda}, \hat{{s}}) & = \nu_k (\lambda^0, {s}^0)+(1-\nu_k)(\lambda^*, {s}^*),
    \label{def_hat_lambda_s}
  \end{align}
\end{subequations}
where $\left(x^*, \lambda^*, s^*\right) \in \SolSet$ is a solution that attains the minimum of \eqref{definition_omega_o}.

\begin{lemma}
  \label{lemma_upper_bound_of_sum_of_x_s_and_z_s}
  If $\left(x^0, s^0\right)$ is defined by \eqref{condition_initial_and_optim_point}, then
  \begin{equation}
    \label{eq_lemma_omega_o_and_l_lower}
    \tau_k \nu_k^\tau (x^0)^\top s^k + \nu_k (s^0)^\top z^k
    \le \left(2 \frac{\omega^o}{1 - \beta} + 1\right) (z^k)^{\top} s^k
  \end{equation}
  for all $k \ge 0$.
\end{lemma}
The proof of Lemma~\ref{lemma_upper_bound_of_sum_of_x_s_and_z_s} can be done with similar steps in \cite[Lemma~6.4]{yang2020arc}.
However,
since the coefficients of \eqref{eq_lemma_omega_o_and_l_lower} are slightly different from those in \cite{yang2020arc},
we include the proof here.
\begin{proof}
  Let $(\hat{x}, \hat{\lambda}, \hat{s})$ be defined by \eqref{def_hat_x_lambda_s},
  then from \eqref{main_residual_norm_decreasing},
  \begin{align*}
    A (\hat{x} - z^k) & = \tau_k \nu_k^\tau r_b(x^0) + (1 - \tau_k \nu_k^\tau) r_b(x^*) - r_b(z^k) \\
    & = \tau_k \nu_k^\tau r_b(x^0) - \tau_k \nu_k^\tau r_b(x^0) = 0 \\
    \hat{s}-s^k & = c-A^{\top} \hat{\lambda}-(c-A^{\top} \lambda^k) = -A^{\top}\left(\hat{\lambda}-\lambda^k\right),
  \end{align*}
  so we can get
  \begin{equation}
    \left(\hat{x}-z^k\right)^\top\left(\hat{s}-s^k\right)
    = -\left(A(\hat{x}-z^k)\right)^\top\left(\hat{\lambda}-\lambda^k\right)
    = 0. \label{equality_zero_product_of_feasible_x_s}
  \end{equation}
  Therefore,  we can derive
  \begin{align}
    & \left(\tau_k \nu_k^\tau x^0+\left(1 - \tau_k \nu_k^\tau \right) x^*\right)^\top {s}^k+\left(\nu_k s^0+\left(1-\nu_k\right) s^*\right)^\top z^k \notag \\
    = & \left(\tau_k \nu_k^\tau x^0+\left(1 - \tau_k \nu_k^\tau \right) x^*\right)^\top\left(\nu_k s^0+\left(1-\nu_k\right) s^*\right) + (z^k)^\top s^k.
    \label{upper_temp}
  \end{align}

  From $\nu_k \in[0,1]$ and $\left(x^*, s^*\right) \ge 0$,
  we can get
  $$ \begin{array}{ll}
    \tau_k \nu_k^\tau (x^0)^\top s^k + \nu_k (s^0)^\top z^k
    & \le \left(\tau_k \nu_k^\tau x^0+\left(1 - \tau_k \nu_k^\tau \right) x^*\right)^\top s^k+\left(\nu_k s^0+\left(1-\nu_k\right) s^*\right)^\top z^k \\
    {[\because \eqref{upper_temp})}
    & = \left(\tau_k \nu_k^\tau x^0+\left(1 - \tau_k \nu_k^\tau\right) x^*\right)^\top\left(\nu_k s^0+\left(1-\nu_k\right) s^*\right)+ (z^k)^\top s^k \\
    {\left[\because  (x^*)^\top s^*=0\right]}
    & =\tau_k \nu_k^\tau \nu_k (x^0)^\top s^0 + \tau_k \nu_k^\tau (1-\nu_k)(x^0)^\top s^* + (1-\tau_k \nu_k^\tau) \nu_k (x^*)^\top s^0 \\
    & \quad + (z^k)^\top s^k \\
    {[\because \eqref{upper_x_s_solution_product}]}
    & \le \tau_k \nu_k^\tau \nu_k (x^0)^\top s^0+ (\tau_k \nu_k^\tau (1-\nu_k) + (1-\tau_k \nu_k^\tau) \nu_k) \omega^o (x^0)^\top s^0 \\
    & \quad + (z^k)^\top s^k \\
    {\left[\because \omega^o \ge 1\right]} & \le (\tau_k \nu_k^\tau \nu_k + \tau_k \nu_k^\tau (1-\nu_k) + (1-\tau_k \nu_k^\tau) \nu_k) \omega^o (x^0)^\top s^0+ (z^k)^\top s^k \\
    {[\because 0 \le \tau_k \le 1 \text{ and } \eqref{def_mus}]} & \le 2 \nu_k \omega^o (x^0)^\top s^0 + (z^k)^\top s^k \\
    {[\because \eqref{mu_decreasing}]} & = 2\omega^o (x^k)^\top s^k + (z^k)^\top s^k \\
    {[\because \eqref{upper_and_lower_z_s_by_x_s}]} & \le \left(2 \frac{\omega^o}{1 - \beta} + 1\right) (z^k)^\top s^k.
  \end{array} $$
  This completes the proof.
\end{proof}

For the latter discussions,
let $D^k=\left(Z^k\right)^{\frac{1}{2}}\left({S}^k\right)^{-\frac{1}{2}}$.
We introduce the following lemma to prove that $\norm{\left(D^k\right)^{-1} \dot{z}}$ and $\norm{\left(D^k\right) \dot{s}}$ have upper bounds.
\begin{lemma}
  \label{lemma_delta_z_s_upper_with_l_delta_lambda}
  For $i=1,2,3$, let $\left(\delta {z}^i, \delta \lambda^i, \delta {s}^i\right)$ be the solution of
  \begin{equation}
    \label{equation_delta_variables}
    \left[\begin{array}{ccc}
      {A} & {0} & {0} \\
      {0} & {A}^\top & {I} \\
      {S^k} & {0} & {Z^k}
    \end{array}\right]\left[\begin{array}{c}
      \delta {z}^i \\
      \delta \lambda^i \\
      \delta {s}^i
    \end{array}\right]=\left[\begin{array}{c}
      0 \\
      0 \\
      r^i
    \end{array}\right]
  \end{equation}
  where $r^1= z^k \circ s^k$, $r^2=-\tau_k \nu_k^\tau s^k \circ(x^0-\bar{x})$
  and $r^3=-\nu_k z^k \circ(s^0-\bar{s})$.
  Then,
  \begin{subequations}
    \label{upper_D_inv_delta_x_i_and_D_delta_s_i}
    \begin{align}
      \norm{\left(D^k\right)^{-1} \delta z^1}, \norm{\left(D^k\right) \delta s^1}
        & \le \norm{\left(D^k\right)^{-1} \delta z^1+\left(D^k\right) \delta s^1} = \norm{({Z^k s^k})^{\frac{1}{2}}} = \sqrt{n \mu_k^z},
        \label{upper_D_inv_delta_x_1_and_D_delta_s_1} \\
      \norm{\left(D^k\right)^{-1} \delta z^2}, \norm{\left(D^k\right) \delta s^2}
        & \le \norm{\left(D^k\right)^{-1} \delta z^2+\left(D^k\right) \delta s^2}
        = \tau_k \nu_k^\tau \norm{\left(D^k\right)^{-1}\left(x^0-\bar{x}\right)}, \label{upper_D_inv_delta_x_2_and_D_delta_s_2} \\
      \norm{\left(D^k\right)^{-1} \delta z^3}, \norm{\left(D^k\right) \delta s^3}
        & \le \norm{\left(D^k\right)^{-1} \delta z^3+\left(D^k\right) \delta s^3}
        = \nu_k \norm{\left(D^k\right)\left(s^0-\bar{s}\right)}. \label{upper_D_inv_delta_x_3_and_D_delta_s_3}
    \end{align}
  \end{subequations}
\end{lemma}

\begin{proof}
  From the first and second rows of \eqref{equation_delta_variables},
  $$
    \left(\left(D^k\right)^{-1} \delta z^i\right)^\top \left(\left(D^k\right) \delta s^i\right) = 0
  $$
  for $i=1,2,3$, therefore,
  \begin{align}
    \norm{\left(D^k\right)^{-1} \delta z^i}^2,\norm{\left(D^k\right) \delta s^i}^2
    & \le \norm{\left(D^k\right)^{-1} \delta z^i}^2+\norm{\left(D^k\right) \delta s^i}^2 \notag \\
    & \le \norm{\left(D^k\right)^{-1} \delta z^i + \left(D^k\right) \delta s^i}^2.
    \label{upper_D_inv_delta_x_and_D_delta_s}
  \end{align}
  Applying
  $(Z^k S^k)^{-1 / 2}(S^k \delta z^i + Z^k \delta s^i)=(Z^k S^k)^{-1 / 2} {r}^i$
  to \eqref{upper_D_inv_delta_x_and_D_delta_s} for $i=1,2,3$, respectively,
  we obtain \eqref{upper_D_inv_delta_x_i_and_D_delta_s_i}.
\end{proof}

In Lemmas~\ref{lemma_upper_norm_D_dot_with_x_s} and \ref{lemma_upper_norm_D_dot} below,
upper bounds of $\norm{\left(D^k\right)^{-1} \dot{z}}$ and $\norm{\left(D^k\right) \dot{s}}$ can be derived.
This allows us to prove the lower bound of $\sin(\alpha_k)$ in the same way as in \cite{yang2018arc}.

\begin{lemma}
  \label{lemma_upper_norm_D_dot_with_x_s}

  Let $(\dot{z}, \dot{s})$ be defined by \eqref{first_derivative}.
  Then,
  \begin{equation}
    \label{upper_max_D_dot_z_and_D_dot_s}
    \begin{aligned}
      \max \left\{\norm{\left(D^k\right)^{-1} \dot{z}},\norm{\left(D^k\right) \dot{s}}\right\}
      \le & \sqrt{n \mu_z^k} + \omega^f \left(2 \frac{\omega^o}{1 - \beta} + 1\right) \frac{\left(z^k\right)^\top s^k}{\min _i \sqrt{z_i^k s_i^k}}.
    \end{aligned}
  \end{equation}
\end{lemma}

\begin{proof}
  Let $(\bar{x}, \bar{\lambda}, \bar{s})$ be the feasible solution of \eqref{problem_main} and \eqref{problem_dual}
  that attains the minimum  of \eqref{definition_omega_f}.
  From \eqref{main_residual_norm_decreasing},
  $$
    A \dot{z} = r_b(z^k) = \tau_k \nu_k^\tau r_b(x^0),
  $$
  we have
  \begin{equation}
    A\left(\dot{z} - \tau_k \nu_k^\tau (x^0 - \bar{x})\right) = 0.
    \label{def_x}
  \end{equation}
  Similarly, since
  $$
    A^\top \dot{\lambda} + \dot{s} = r_c(\lambda^k, s^k) = \nu_k \left(A^\top \lambda^0 + s^0 - c\right)
    = \nu_k \left(A^\top\left(\lambda^0-\bar{\lambda}\right)+\left({s}^0-\bar{s}\right)\right),
  $$
  it holds that
  \begin{align}
    A^\top \left(\dot{\lambda} - \nu_k\left(\lambda^0-\bar{\lambda}\right)\right)+\left(\dot{s} - \nu_k \left({s}^0-\bar{s}\right)\right) = 0. \label{def_lambda_s}
  \end{align}
  From the third row of \eqref{first_derivative}, we also have
  \begin{align}
    & \ s^k \circ \left(\dot{z} - \tau_k \nu_k^\tau \left({x}^0-\bar{x}\right)\right) + z^k \circ \left(\dot{s} - \nu_k \left(s^0-\bar{s}\right)\right) \notag \\
    = & \ z^k \circ s^k - \tau_k \nu_k^\tau s^k \circ\left({x}^0-\bar{{x}}\right) - \nu_k z^k \circ\left({s}^0-\bar{{s}}\right). \label{def_z_s}
  \end{align}
  By combining
  \eqref{def_x}, \eqref{def_lambda_s} and \eqref{def_z_s} into a matrix form,
  we have
  $$
    {\left[\begin{array}{ccc}
      A & {0} & {0} \\
      {0} & A^{\top} & {I} \\
      S^k & {0} & Z^k
    \end{array}\right]\left[\begin{array}{c}
      \dot{z}-\tau_k \nu_k^\tau \left({x}^0-\bar{{x}}\right) \\
      \dot{\lambda}-\nu_k\left(\lambda^0-\bar{\lambda}\right) \\
      \dot{s}-\nu_k\left({s}^0-\bar{{s}}\right)
    \end{array}\right] } \\
    = {\left[\begin{array}{c}
      0 \\
      0 \\
      z^k \circ s^k - \tau_k \nu_k^\tau s^k \circ\left(x^0-\bar{x}\right) - \nu_k z^k \circ\left({s}^0-\bar{s}\right)
    \end{array}\right]. }
  $$
  Putting
  \begin{gather*}
    \delta z^1+\delta z^2+\delta z^3 = \dot{z}-\tau_k \nu_k^\tau \left(x^0-\bar{x}\right)\\
    \delta \lambda^1+\delta \lambda^2+\delta \lambda^3 = \dot{\lambda}-\nu_k\left(\lambda^0-\bar{\lambda}\right)\\
    \delta s^1 + \delta s^2 + \delta s^3 = \dot{s}-\nu_k\left(s^0-\bar{s}\right) \\
    (r^1, r^2, r^3) = \left(z^k \circ s^k, -\tau_k \nu_k^\tau s^k \circ(x^0-\bar{x}), -\nu_k z^k \circ({s}^0-\bar{s})\right),
  \end{gather*}
  into \eqref{equation_delta_variables},
  we apply Lemma~\ref{lemma_delta_z_s_upper_with_l_delta_lambda} to obtain the upper bounds of
  \begin{align}
    \norm{\left(D^k\right)^{-1} \dot{z}}
    & = \norm{\left(D^k\right)^{-1}\left[\delta z^1 + \delta z^2 + \delta z^3 + \tau_k \nu_k^\tau \left(x^0-\bar{x}\right)\right]}, \label{D_inv_dot_z_separated}\\
    \norm{(D^k) \dot{s}} & = \norm{\left(D^k\right)\left[\delta s^1 + \delta s^2 + \delta s^3 + \nu_k\left(s^0-\bar{s}\right)\right]}. \label{D_dot_s_separated}
  \end{align}
  Considering \eqref{equation_delta_variables} with $i=2$, we can get
  $$
    S^k \delta z^2 + Z^k \delta s^2 = r^2 = -\tau_k \nu_k^\tau S^k \left(x^0-\bar{x}\right),
  $$
  and this is equal to
  \begin{equation}
    \label{eq_delta_z_2}
    \delta z^2 = -\tau_k \nu_k^\tau \left(x^0-\bar{x}\right) - \left(D^k\right)^2 \delta s^2 .
  \end{equation}
  Thus, from \eqref{D_inv_dot_z_separated} and \eqref{eq_delta_z_2}, we have
  \begin{equation}
    \label{upper_D_dot_z}
    \begin{array}{rl}
      \norm{\left(D^k\right)^{-1} \dot{z}}
      & = \norm{\left(D^k\right)^{-1} \delta z^1 -\left(D^k\right) \delta s^2+\left(D^k\right)^{-1} \delta z^3} \\
      & \le \norm{\left(D^k\right)^{-1} \delta z^1 }+\norm{\left(D^k\right) \delta s^2 }+\norm{\left(D^k\right)^{-1} \delta z^3}.
    \end{array}
  \end{equation}
  Similarly, for $i=3$, we have
  $$
    S \delta z^3 + Z \delta s^3 = r^3 = -\nu_k Z^k \left(s^0-\bar{s}\right),
  $$
  which is equivalent to
  \begin{equation}
    \label{eq_delat_s_3}
    \delta s^3 = -\nu_k \left(s^0-\bar{s}\right)-(D^k)^{-2} \delta z^3 .
  \end{equation}
  Thus, from \eqref{D_dot_s_separated} and \eqref{eq_delat_s_3}, we have
  \begin{equation}
    \label{upper_D_dot_s}
    \begin{array}{rl}
      \norm{(D^k) \dot{s}}
      & = \norm{\left(D^k\right) \delta s^1 + \left(D^k\right) \delta s^2-\left(D^k\right)^{-1} \delta z^3} \\
      & \le \norm{\left(D^k\right) \delta s^1}+\norm{\left(D^k\right) \delta s^2}+\norm{\left(D^k\right)^{-1} \delta z^3}.
    \end{array}
  \end{equation}

  From \eqref{upper_D_inv_delta_x_1_and_D_delta_s_1},
  \begin{equation}
    \label{upper_D_delta_s_1_and_D_delta_z_1}
    \begin{array}{rl}
      \norm{(D^k) \delta s^1 }, \norm{\left(D^k\right)^{-1} \delta z^1 }
      & \le \sqrt{n \mu_z^k}.
    \end{array}
  \end{equation}

  Using \eqref{upper_D_inv_delta_x_2_and_D_delta_s_2}, it holds that
  \begin{align}
    \norm{\left(D^k\right) \delta s^2} & \le \tau_k \nu_k^\tau \norm{\left(D^k\right)^{-1}\left(x^0-\bar{x}\right)} \notag \\
    & \le \frac{\tau_k \nu_k^\tau}{\min \sqrt{z^k_i s^k_i}}\norm{S^k\left(x^0-\bar{x}\right)} \notag \\
    {[\because \eqref{upper_feasible_solution}]} & \le \frac{\tau_k \nu_k^\tau \omega^f}{\min \sqrt{z^k_i s^k_i}} \norm{S^k x^0} \notag \\
    {[\because (\norm{\cdot}_2 \le \norm{\cdot}_1)]} & \le \frac{\tau_k \nu_k^\tau \omega^f}{\min \sqrt{z^k_i s^k_i}} \left((s^k)^\top x^0\right).
    \label{upper_D_delta_s_2}
  \end{align}
  Using \eqref{upper_D_inv_delta_x_3_and_D_delta_s_3}, we have
  \begin{align}
    \norm{\left(D^k\right)^{-1} \delta z^3} & \le \nu_k \norm{\left(D^k\right)\left(s^0-\bar{s}\right)} \notag \\
    & \le \frac{\nu_k}{\min \sqrt{z^k_i s^k_i}}\norm{Z^k\left(s^0-\bar{s}\right)} \notag \\
    {[\because \eqref{upper_feasible_solution}]} & \le \frac{\nu_k \omega^f}{\min \sqrt{z^k_i s^k_i}}\norm{Z^k s^0} \notag \\
    {[\because (\norm{\cdot}_2 \le \norm{\cdot}_1)]} & \le \frac{\nu_k \omega^f}{\min \sqrt{z^k_i s^k_i}} \left((z^k)^\top s^0\right).
    \label{upper_D_delta_z_3}
  \end{align}

  Therefore, from \eqref{upper_D_delta_s_2}, \eqref{upper_D_delta_z_3} and \eqref{eq_lemma_omega_o_and_l_lower},
  \begin{align}
    \norm{\left(D^k\right) \delta s^2}+\norm{\left(D^k\right)^{-1} \delta z^3}
    & \le \frac{\omega^f}{\min_i \sqrt{z^k_i s^k_i}}\left(\tau_k \nu_k^\tau (s^k)^\top x^0 + \nu_k (z^k)^\top s^0\right) \notag \\
    & \le \frac{\omega^f}{\min_i \sqrt{z_i^k s_i^k}} \left(2 \frac{\omega^o}{1 - \beta} + 1\right) (z^k)^\top s^k.
    \label{ddd_omega_f}
  \end{align}

  Therefore, from \eqref{upper_D_dot_z} and \eqref{upper_D_dot_s},
  \begin{align*}
    \max \left\{\norm{\left(D^k\right)^{-1} \dot{z}},\norm{\left(D^k\right) \dot{s}}\right\}
    \le & \max \left\{\norm{\left(D^k\right)^{-1} \delta z^1},\norm{\left(D^k\right) \delta s^1}\right\} \\
    & + \norm{\left(D^k\right) \delta s^2}+\norm{\left(D^k\right)^{-1} \delta z^3} \\
    {[\because \eqref{upper_D_delta_s_1_and_D_delta_z_1}]}
    \le & \sqrt{n \mu_z^k} + \norm{\left(D^k\right) \delta s^2}+\norm{\left(D^k\right)^{-1} \delta z^3} \\
    {[\because \eqref{ddd_omega_f}]} \le & \sqrt{n \mu_z^k}
    + \omega^f \left(2 \frac{\omega^o}{1 - \beta} + 1\right) \frac{\left(z^k\right)^\top s^k}{\min_i \sqrt{z_i^k s_i^k}}.
  \end{align*}
\end{proof}

The upper bounds derived in Lemma~\ref{lemma_upper_bound_of_sum_of_x_s_and_z_s} and
\ref{lemma_upper_norm_D_dot_with_x_s} depend on the value of $\beta$.
If $\beta$ is very close to 1, the upper bounds would be very large.
Conversely, if it is close to 0, the momentum term has only a small effect.
For the proof, we use $\beta \in (0, 1)$,
and we will discuss the effect of $\beta$ numerically through numerical results in
Section~\ref{subsection_comparison_between_algo_1_and_2} and \ref{subsection_comparison_beta}.

Lemma~\ref{lemma_upper_norm_D_dot_with_x_s} leads to the following lemma.

\begin{lemma}
  \label{lemma_upper_norm_D_dot}
  Let the sequence $\{(x^k, \lambda^k, s^k)\}$ be generated by Algorithm~\ref{algo_arc_restarting_for_proof}
  and let $(\dot{z}, \dot{\lambda}, \dot{s})$ be obtained from (\ref{first_derivative}).
  Then, there is a positive constant $C_0$ satisfying
  \begin{equation}
    \label{upper_max_D_dot_z_and_D_dot_s_independent_of_n}
    \max \left\{\norm{\left(D^k\right)^{-1} \dot{z}},\norm{\left(D^k\right) \dot{s}}\right\}
    \le C_0 \sqrt{n\left(z^{k}\right)^\top s^k},
  \end{equation}
  and independent of $n$.
\end{lemma}

Lemma~\ref{lemma_upper_norm_D_dot} can be proved by replacing $x^k$ and $\dot{x}$ in \cite[Lemma~6.7]{yang2020arc} by $z^k$ and $\dot{z}$.
In this case, $C_0$ is taken such that
$$
  C_0 \ge 1 + \omega^f \frac{\left(2 \frac{\omega^o}{1 - \beta} + 1\right)}{\sqrt{(1-\theta)}}.
$$

From Lemma~\ref{lemma_upper_norm_D_dot},
we can give an upper bound on the norm of the search directions.
\begin{lemma}
  \label{lemma_upper_derivative_norms}
  Let $(\dot{z}, \dot{\lambda}, \dot{s})$ and $(\ddot{z}, \ddot{\lambda}, \ddot{s})$ be calculated by
  \eqref{first_derivative} and \eqref{second_derivative}.
  Then, there are positive constants
  $C_{1}, C_{2}, C_{3}$, and $C_{4}$ satisfying
  \begin{align}
    &\norm{\dot{z} \circ \dot{s}} \le C_{1} n^{2} \mu_k^z, \label{upper_first_derivative_Hadamard_product}\\
    &\norm{\ddot{z} \circ \ddot{s}} \le C_{2} n^{4} \mu_k^z, \\
    &\max \left\{\norm{\left(D^{k}\right)^{-1} \ddot{z}}, \norm{\left(D^{k}\right) \ddot{s}}\right\}
      \le C_{3} n^{2} \sqrt{\mu_k^z}, \label{max_D_ddot_z_and_D_ddot_s_less_than_constant} \\
    &\max \{\norm{\ddot{z} \circ \dot{s}}, \norm{\dot{z} \circ \ddot{s}}\} \le C_{4} n^{3} \mu_k^z.
  \end{align}
In addition, $C_{1}, C_{2}, C_{3}$, and $C_{4}$ are independent of $n$.
\end{lemma}

We can apply the same discussion as Yang~\cite[Lemma~4.4]{yang2018arc} to prove the above lemma,
by replacing $x$ and $\mu_k$ with $z$ and $\mu_k^z$, respectively.
The above lemma allows us to make an estimation regarding the lower bound of $\sin \left(\alpha_{k}\right)$.

\begin{lemma}
  \label{lemma_lower_bound_sin_alpha}
  Let the sequence $\{(x^k, \lambda^k, s^k)\}$ be generated by Algorithm~\ref{algo_arc_restarting_for_proof}.
  Then, $\alpha_{k}$ satisfies the following inequality:
  $$
    \sin \left(\alpha_{k}\right) \ge \frac{\theta}{2 C n},
  $$
  where $C=\max \left\{1, C_{4}^{\frac{1}{3}},\left(C_{1}+C_{2}\right)^{\frac{1}{4}}\right\}$.
\end{lemma}

The proof of Lemma~\ref{lemma_lower_bound_sin_alpha} can be done
in a similar way to \cite[Lemma~4.5]{yang2018arc} with a replacement $\mu_k$ by $\mu_k^z$.

Lemma~\ref{lemma_lower_bound_sin_alpha} enables to establish polynomial-time complexity from the same argument as in \cite{yang2018arc}.
We consider the following stopping criterion at Step 1 of Algorithm~\ref{algo_arc_restarting_for_proof}:
\begin{equation}
  \label{stopping_critera}
  \mu_k \le \epsilon, \quad
  \norm{r_b(x^k)} \le \frac{\left\|r_b(x^0)\right\|}{\mu_0} \epsilon, \quad
  \norm{r_c(\lambda^k, s^k)} \le \frac{\left\|r_c(\lambda^0, s^0)\right\|}{\mu_0} \epsilon.
\end{equation}

\begin{theorem}
  \label{theorem_polynomiality}

  Let $\{(x^k, \lambda^k, s^k)\}$ be generated by Algorithm~\ref{algo_arc_restarting_for_proof}
  with an initial point given by \eqref{condition_initial_and_optim_point} and
  $L=\ln(\mu^0 / \epsilon)$ for a given $\epsilon>0$.
  Then,
  the algorithm will terminate with $(x^k, \lambda^k, s^k)$
  satisfying \eqref{stopping_critera}
  in at most $\Order(n L)$ iterations.
\end{theorem}
The proof of Theorem~\ref{theorem_polynomiality} is the essentially same as Wright~\cite[Theorem 3.2]{wright1997primal} as well as \cite[Theorem~4.1]{yang2018arc}.
Therefore,
Algorithm~\ref{algo_arc_restarting_for_proof} reaches a solution up to the stopping threshold $\epsilon$.

\subsection{Proof for Proposition~\ref{proposition_s_bounded_below_away_from_zero}}
Lastly, we give the proof of Proposition~\ref{proposition_s_bounded_below_away_from_zero} for nonsingularity of \eqref{matrix_need_nonsingularity}.
For this purpose,
since $(x^k, \lambda^k, s^k)$ is in the neighborhood $\Neighborhood(\theta)$,
we prove that there is a lower bound that is greater than 0 on $s^k$ by finding an upper bound for $x^k$.

\begin{lemma}
  \label{lemma_upper_x}
  There is a positive constant $\kappa$ such that
  $$
    \norm{x^k}_1 \le \kappa
  $$
  for each $k$.
\end{lemma}

The proof of Lemma~\ref{lemma_upper_x} is given similarly to \cite[Lemma 6.3]{wright1997primal}.

\begin{proof}
  Similarly to \eqref{equality_zero_product_of_feasible_x_s}, we obtain
  $$
    (\nu_k^\tau x^0 + (1 - \nu_k^\tau) x^* - x^k)^\top (\hat{s} - s^k) = 0,
  $$
  where $\hat{s}$ is defined by \eqref{def_hat_lambda_s}.
  Therefore,
  \begin{align*}
    \nu_k (x^k)^\top s^0 + \nu_k^\tau (x^0)^\top s^k
    & = \nu_k^\tau \nu_k (x^0)^\top s^0 + \nu_k^\tau (1 - \nu_k) (x^0)^\top s^* + \nu_k (1 - \nu_k^\tau) (x^*)^\top s^0 + (x^k)^\top s^k \\
    & \quad -(1 - \nu_k^\tau)(x^*)^\top s^k -(1 - \nu_k)(x^k)^\top s^* + (1 - \nu_k)(1 - \nu_k^\tau) (x^*)^\top s^*\\
    {\left[\because  (x^*)^\top s^*=0\right]}
    & = \nu_k^\tau \nu_k (x^0)^\top s^0 + \nu_k^\tau (1 - \nu_k) (x^0)^\top s^* + \nu_k (1 - \nu_k^\tau) (x^*)^\top s^0 \\
    & \quad + (x^k)^\top s^k - (1 - \nu_k^\tau)(x^*)^\top s^k -(1 - \nu_k)(x^k)^\top s^* \\
    {[\because x^k, s^k, x^*, s^* \ge 0]}
    & \le \nu_k^\tau \nu_k (x^0)^\top s^0 + \nu_k^\tau (1 - \nu_k) (x^0)^\top s^* + \nu_k (1 - \nu_k^\tau) (x^*)^\top s^0 + (x^k)^\top s^k.
  \end{align*}
  Let us now define the constant $\xi$ by
  $$
    \xi = \min_i s^0_i.
  $$
  Because $s^0 > 0$, we have $\xi > 0$,
thus,
  $$
    \xi \norm{x^k}_1 \le \min_i (s_i^0) \norm{x^k}_1 \le (x^k)^\top s^0.
  $$
  From $\nu_k^\tau \ge 0$, $x^0, s^k \ge 0$, we have
  $$
    \nu_k (x^k)^\top s^0 + \nu_k^\tau (x^0)^\top s^k \ge \nu_k (x^k)^\top s^0 \ge \nu_k \xi \norm{x^k}_1,
  $$
  so it holds that
  $$
    \nu_k \xi \norm{x^k}_1
    \le \nu_k^\tau \nu_k n \mu_0 + \nu_k^\tau (1 - \nu_k) \norm{x^0}_\infty \norm{s^*}_1 + \nu_k (1 - \nu_k^\tau) \norm{x^*}_1 \norm{s^0}_\infty + n \mu_k.
  $$
  Furthermore,  \eqref{mu_decreasing} and  $\nu_k^\tau \le \nu_k \le 1$ lead to
  \begin{align*}
    \xi \norm{x^k}_1
    & \le \nu_k^\tau n \mu_0 + \frac{\nu_k^\tau}{\nu_k}(1 - \nu_k) \norm{x^0}_\infty \norm{s^*}_1 + (1 - \nu_k^\tau) \norm{x^*}_1 \norm{s^0}_\infty + n \mu_0\\
    & \le 2n \mu_0 + (1 - \nu_k) \norm{x^0}_\infty \norm{s^*}_1 + (1 - \nu_k^\tau) \norm{x^*}_1 \norm{s^0}_\infty \\
    & \le 2n \mu_0 + \norm{x^0}_\infty \norm{s^*}_1 + \norm{x^*}_1 \norm{s^0}_\infty.
  \end{align*}
  The proof is completed by setting
  $$
    \kappa = \xi^{-1} \left(2n \mu_0 + \norm{x^0}_\infty \norm{s^*}_1 + \norm{x^*}_1 \norm{s^0}_\infty\right).
  $$
\end{proof}

Based on the above, we prove the Proposition~\ref{proposition_s_bounded_below_away_from_zero}.

\begin{proof}[Proof of Proposition~\ref{proposition_s_bounded_below_away_from_zero}]
  From $(x^k, \lambda^k, s^k) \in \Neighborhood(\theta)$,
  we obtain
  $$
    x^k_i s^k_i \ge (1 - \theta) \mu_k.
  $$
  Since there is the constant $\kappa$ such that $x_i^k \le \kappa$ from Lemma~\ref{lemma_upper_x} and $x^k>0$,
  we have
  \begin{equation}
    \label{lower_bound_of_s}
    s^k_i \ge (1 - \theta) \mu_k \frac{1}{x^k_i} \ge (1 - \theta) \mu_k \frac{1}{\kappa} > 0.
  \end{equation}
  Therefore, Proposition~\ref{proposition_s_bounded_below_away_from_zero}
holds by $L_0 = (1-\theta) / \kappa$.
\end{proof}
\section{Numerical experiments}
\label{section_numerical_experiments}

We conducted numerical experiments to compare the proposed method
with the existing methods, i.e., arc-search method and line-search method.
In addition, we also investigate the effect of the restarting parameter $\beta$.
The numerical experiments were executed on a Linux server with Opteron 4386 (3.10GHz), 16 cores and 128~GB RAM.
We implemented the methods with Python 3.10.9.

In the following,
an existing arc-search interior-point method \cite{Yang2017} is called ``Arc-search'' and
an existing Mehortra type line-search interior-point method \cite{Mehrotra1992} ``Line-search''.
Both of these implementations are built upon the algorithms described in the referenced papers.

\subsection{Test problems}
For the test problems, we used the Netlib test repository~\cite{browne1995netlib}.
We chose 85 problems in the experiment
\footnote{
	We excluded
	the following problems, since they were too large or some infeasibility was detected at a preprocessing stage:
	80BAU3B,
	BLEND,
	CRE-B, CRE-D,
	DEGEN2,
	DFL001,
	E226,
	FIT2D, FIT2P,
	FORPLAN,
	GFRD-PNC, GROW7, GROW15,
	GROW22,
	KEN-11,
	KEN-13, KEN-18,
	NESM,
	OSA-07, OSA-14,
	OSA-30, OSA-60,
	PDS-06, PDS-10, PDS-20,
	QAP15,
	SCORPION,
	SIERRA,
	STOCFOR3.
}.

The Netlib instances are given in the following format:
\begin{align*}
	\min _{x_b} \quad & c_b^\top x_b \\
	\text{subject to } & A_E x_b = b_E, \
	A_G x_b \ge b_G, \
	A_L x_b \le b_L, \
	x_b \ge b_{LO}, \
	x_b \le b_{UP}.
\end{align*}
(There exist cases where $b_{LO} = b_{UP}$.)
To transform it into the form of (\ref{problem_main}), we employ the following setting:
\begin{equation*}
	x = \left[\begin{array}{c}
		x_b - b_{LO} \\ s_G \\ s_L \\ s_B
	\end{array}\right], \quad
	A = \left[\begin{array}{cccc}
		A_E & 0 & 0 & 0 \\
		A_G & -I & 0 & 0 \\
		A_L & 0 & I & 0 \\
		I_B & 0 & 0 & I
	\end{array}\right], \quad
	b = \left[\begin{array}{c}
		b_E - A_E b_{LO} \\
		b_G - A_G b_{LO} \\
		b_L - A_L b_{LO} \\
		b_{UP} - b_{LO}
	\end{array}\right], \quad
	c = \left[\begin{array}{c}
		c_b \\ 0 \\ 0 \\ 0
	\end{array}\right]
\end{equation*}
where $s_G$, $s_L$, and $s_B$ are slack variables for inequality constraints such that
\begin{align*}
	A_G x_b - s_G  = b_G, \
	A_L x_b + s_L  = b_L, \
	x_b + s_B  = b_{UP}.
\end{align*}
Though another setting like $x_b - s_{LO} = b_{LO}$ can be considered, the above setting has fewer constraints.

In addition, we apply the preprocessing described in \cite[Section 4.2]{Yang2017} to each problem.
Therefore,
the size $(n,m)$ of each problem in this paper is not always the same as the size reported at the Netlib webpage~\cite{browne1995netlib}.
The variable size $n$ ranges from 51 to 9,253,
and the number of constraints $m$ from 27 to 4,523.

\subsection{The modified algorithm}
\label{subsection_algorithms_for_comparison}
To improve practical performance of Algorithm~\ref{algo_arc_restarting_for_proof},
we modify it into  Algorithm~\ref{algo_arc_restarting_for_calculation}
with the Mehrotra type implementation~\cite{Mehrotra1992} similarly to \cite[Algorithm 3.1]{Yang2017}.
The comparison of numerical experiments in Section~\ref{subsection_comparison_between_algo_1_and_2} will show
that Algorithm~\ref{algo_arc_restarting_for_calculation} is numerically superior to Algorithm~\ref{algo_arc_restarting_for_proof}.
In the following, we discuss the main differences.

Based on the concept of Mehrotra type implementation, we slightly perturb the second derivative \eqref{second_derivative}
to prevent the step size from diminishing to zero:
\begin{equation}
	\label{second_derivative_for_calculation}
	\begin{bmatrix}
		A & 0 & 0 \\
		0 & A^\top & I \\
		S^k & 0 & Z^k
	\end{bmatrix} \left[\begin{array}{c}
		\ddot{z} \\
		\ddot{\lambda} \\
		\ddot{s}
	\end{array}\right]
	=
	\left[\begin{array}{c}
		0 \\
		0 \\
		\sigma_k \mu_k^z e - 2 \dot{z} \circ \dot{s}
	\end{array}\right].
\end{equation}
Here $\sigma_k \in (0, 0.5] $ is called the centering parameter,
and is selected in the same way as Mehrotra~\cite{Mehrotra1992}.

As for the setting of $\beta_k$,
we use two formulas of computation: the formula in \eqref{def_beta_for_proof} and
its simplified formula:
\begin{equation}
	\label{def_beta_for_calculation}
	\beta_k = \frac{\beta}{\norm{ X_k^{-1}\delta(x^k) }_\infty}.
\end{equation}
Although the term $\abs{\frac{r_b(x^k)_j}{r_b(x^k)_j - r_b(x^{k-1})_j}}$ was used in the proof,
we mainly use \eqref{def_beta_for_calculation} for setting $\beta_k$
because \eqref{def_beta_for_calculation} was more efficient in preliminary numerical experiments.
A comparison of the difference in $\beta_k$ will be reported
in Section~\ref{subsection_comparison_setting_of_beta_k}.

In addition, Algorithm~\ref{algo_arc_restarting_for_calculation} does not use the neighborhood $\Neighborhood(\theta)$, thus
$z^k$ is evaluated as follows:
\begin{equation}
	\label{def_restart_point_for_calculation}
	z^k = x^k + \beta_k \delta(x^k).
\end{equation}
In other words,
the momentum term is employed in all the iterations.

\begin{algorithm}[H]
  \caption{Mehrotra type arc-search interior-point method with Nesterov's restarting}
  \label{algo_arc_restarting_for_calculation}
  \begin{algorithmic}[1]
		\renewcommand{\algorithmicrequire}{\textbf{Input:}}
    \renewcommand{\algorithmicensure}{\textbf{Output:}}
    \Require an initial point $(x^0 > 0, \lambda^0, s^0 > 0)$,
			restarting parameter $\beta \in (0, 1)$.
    \State If $(x^k, \lambda^k, s^k)$ satisfies a stopping criterion, output $(x^k, \lambda^k, s^k)$ as an optimal solution and stop.
    \State Set $\beta_k$.
    \State Set $z^k$ using \eqref{def_restart_point_for_calculation} and $\mu_k^z$ using \eqref{def_duality_measure_with_restarting_point}.
    \State Calculate $(\dot{z}, \dot{\lambda}, \dot{s})$ using \eqref{first_derivative}, and set
			\begin{align*}
				\alpha^z_a &:=\arg \max \{\alpha \in[0,1] \mid z^k - \alpha \dot{z} \ge 0\} \\
				\alpha^s_a &:=\arg \max \{\alpha \in[0,1] \mid s^k - \alpha \dot{s} \ge 0\}.
			\end{align*}
    \State Calculate $\mu_a=\left(z^k - \alpha^z_a \dot{z}\right)^\top \left(s^k - \alpha^s_a \dot{s} \right) / n$ and
      the centering parameter $\sigma_k=\left(\mu_a / \mu_k^z\right)^{3}$.
    \State Calculate $(\ddot{z}, \ddot{\lambda}, \ddot{s})$ using (\ref{second_derivative_for_calculation}) and set
      \begin{subequations}
				\label{def_alpha_by_variables}
        \begin{align}
          \alpha^z_{\max} &=\arg \max \left\{\alpha \in\left[0, \frac{\pi}{2}\right] \mid z^{k}-\dot{z} \sin (\alpha)+\ddot{z}(1-\cos (\alpha)) \ge 0\right\} \\
          \alpha^s_{\max} &=\arg \max \left\{\alpha \in\left[0, \frac{\pi}{2}\right] \mid s^{k}-\dot{s} \sin (\alpha)+\ddot{s}(1-\cos (\alpha)) \ge 0\right\}.
        \end{align}
      \end{subequations}
		\State If
			\begin{subequations}
				\begin{align*}
					x^{k+1}_{\max} & = z^{k}-\dot{z} \sin \left(\alpha^z_{\max}\right)+\ddot{z}\left(1-\cos \left(\alpha^z_{\max}\right)\right) \ge 0\\
					\lambda^{k+1}_{\max} & = \lambda^{k}-\dot{\lambda} \sin \left(\alpha^s_{\max}\right)+\ddot{\lambda}\left(1-\cos \left(\alpha^s_{\max}\right)\right) \\
					s^{k+1}_{\max} & = s^{k}-\dot{s} \sin \left(\alpha^s_{\max}\right)+\ddot{s}\left(1-\cos \left(\alpha^s_{\max}\right)\right) \ge 0
				\end{align*}
			\end{subequations}
			satisfies the stopping criterion, output $(x^{k+1}_{\max}, \lambda^{k+1}_{\max}, s^{k+1}_{\max})$ as an optimal solution and stop.
    \State Set step scaling factor $\gamma_k \in (0, 1)$,
			scale step size $\alpha_k^z = \gamma_k \alpha^z_{\max}$, $\alpha_k^s = \gamma_k \alpha^s_{\max}$ and update
      \begin{subequations}
        \label{update_formula_restarting}
        \begin{align}
          x^{k+1} & = z^{k}-\dot{z} \sin \left(\alpha_k^z\right)+\ddot{z}\left(1-\cos \left(\alpha_k^z\right)\right) > 0 \\
          \lambda^{k+1} & = \lambda^{k}-\dot{\lambda} \sin \left(\alpha_k^s\right)+\ddot{\lambda}\left(1-\cos \left(\alpha_k^s\right)\right) \\
          s^{k+1} & = s^{k}-\dot{s} \sin \left(\alpha_k^s\right)+\ddot{s}\left(1-\cos \left(\alpha_k^s\right)\right) > 0.
        \end{align}
      \end{subequations}
    \State Set $\mu_{k+1} = (x^{k+1})^\top s^{k+1} / n$. $k \leftarrow k+1$. Go back to Step 1.
  \end{algorithmic}
\end{algorithm}

\subsection{Implementation details}
In the following, we discuss more details on implementation.
In Algorithm~\ref{algo_arc_restarting_for_proof}, we set $\theta=0.25$.

\subsubsection{Initial points}\label{subsection_initial_point}
In Algorithm~\ref{algo_arc_restarting_for_proof}, the components of $x^0$ and $s^0$ are all set to 100 and those of $\lambda^0$ is all set to 0
to ensure that the initial point $(x^0, \lambda^0, s^0)$ is in neighborhood $\Neighborhood(\theta)$.
In contrast,
the initial point is set the same as Yang~\cite[Section 4.1]{Yang2017} in Algorithm~\ref{algo_arc_restarting_for_calculation},
Arc-search and Line-search.

\subsubsection{Step size}
In Algorithm~\ref{algo_arc_restarting_for_proof},
since it is difficult to obtain the solution of \eqref{condition_step_size_for_proof} analytically,
Armijo's rule~\cite{wright1997primal} is employed to determine an actual step size $\alpha_k$ that satisfies \eqref{condition_step_size_for_proof}.
In Algorithm~\ref{algo_arc_restarting_for_calculation},
we decide $\alpha^z_{\max}$ and $\alpha^s_{\max}$ by \eqref{def_alpha_by_variables} on the same strategy as Yang~\cite[Section 4.8]{Yang2017}.

The step scaling factor $\gamma_k$ in Algorithm~\ref{algo_arc_restarting_for_calculation} is calculated as
$\gamma_k = 0.9$ for guaranteeing the positiveness of $x^k$ and $s^k$.
Although Yang~\cite[Section 7.3.9]{yang2020arc} proposed the scaling factor of $\gamma_k = 1 - e^{-(k+2)}$,
it is not robust for numerical errors in our implementation because it converges to 1 too quickly.

\subsubsection{Stopping criteria}\label{section_stopping_criteria}
In Theorem~\ref{theorem_polynomiality},
we used \eqref{stopping_critera} as the stopping criterion for Algorithm~\ref{algo_arc_restarting_for_proof}.
However, the part $\mu_k \le \epsilon$ does not take the magnitude
of the  data into consideration, thus, it is not practical
especially when the magnitude of the optimal values is relatively large.
In addition,
\eqref{stopping_critera} depends on the initial point, though
we set different initial points for Algorithms~\ref{algo_arc_restarting_for_proof} and \ref{algo_arc_restarting_for_calculation} as described in Section~\ref{subsection_initial_point}.
Therefore,
we employed the following stopping criterion in the numerical experiment:
\begin{equation}
	\label{condition_solved}
	\max \left\{
		\frac{\norm{r_b(x^k)}}{\max \{1,\norm{b}\}},
		\frac{\norm{r_c(\lambda^k, s^k)}}{\max \{1,\norm{c}\}},
		\frac{\mu_k}{\max \left\{1,\norm{c^\top x^k},\norm{b^\top \lambda^k}\right\}}
	\right\} < \epsilon,
\end{equation}
where we set the threshold $\epsilon = 10^{-7}$.

In addition, we stopped the iteration immaturely when one of three conditions was detected;
(i) the iteration count $k$ reached 100,
(ii) the step size $\alpha_k$ was too small like $\alpha_k < 10^{-7}$,
or (iii) the linear systems \eqref{first_derivative},
\eqref{second_derivative} or \eqref{second_derivative_for_calculation} could not be solved accurately due to numerical errors.

\subsection{Numerical results}
We report numerical results as follows.
In Section~\ref{subsection_comparison_between_algo_1_and_2},
we compare Algorithms~\ref{algo_arc_restarting_for_proof} and \ref{algo_arc_restarting_for_calculation},
and show that Algorithm~\ref{algo_arc_restarting_for_calculation} is superior to Algorithm~\ref{algo_arc_restarting_for_proof} in the numerical experiments.
Therefore, we use Algorithm~\ref{algo_arc_restarting_for_calculation}
in Section~\ref{subsection_comparison_setting_of_beta_k} and thereafter.
Section~\ref{subsection_comparison_setting_of_beta_k} compares
the settings \eqref{def_beta_for_proof} or \eqref{def_beta_for_calculation}
for $\beta_k$ in Algorithm~\ref{algo_arc_restarting_for_calculation},
and
Section~\ref{subsection_comparison_beta} evaluates the effect of
the value of $\beta$.
Lastly,
Section~\ref{subsection_comparison_with_existing_methods}
compares the numerical performance of Algorithm~\ref{algo_arc_restarting_for_calculation}
with the existing methods, using the  formula \eqref{def_beta_for_calculation}
chosen from Section~\ref{subsection_comparison_setting_of_beta_k}
and
the best $\beta$ from Section~\ref{subsection_comparison_beta}.

\subsubsection{Comparison between Algorithm~\ref{algo_arc_restarting_for_proof} and Algorithm~\ref{algo_arc_restarting_for_calculation}}
\label{subsection_comparison_between_algo_1_and_2}
Table~\ref{table_comparison_solvable_problem_number} reports the number of Netlib problems solved
by Algorithms~\ref{algo_arc_restarting_for_proof} and \ref{algo_arc_restarting_for_calculation} with different settings of $\beta$.
The value of
$\beta_k$ is set by \eqref{def_beta_for_calculation} except a case of $\beta = 0.5$ in Algorithm~\ref{algo_arc_restarting_for_calculation}
that investigates both \eqref{def_beta_for_proof} and \eqref{def_beta_for_calculation}.
The table also shows the number of unsolved problems due to
(i) the iteration limit, (ii) the diminishing step length or (iii) numerical errors---see Section~\ref{section_stopping_criteria}.
\begin{table}[ht]
  \caption{The numbers of problems solved by
	Algorithm~\ref{algo_arc_restarting_for_proof} and Algorithm~\ref{algo_arc_restarting_for_calculation}}
  \label{table_comparison_solvable_problem_number}
  \centering
  \begin{tabular}{llrrr}
		\hline
		Solver & Settings &
		\begin{tabular}{c} Solved \\  \end{tabular} &
		\begin{tabular}{c} Unsolved \\  (i)  \end{tabular}  &
		\begin{tabular}{c} Unsolved \\ (ii) or (iii)  \end{tabular}
		\\
		\hline \hline
		Algorithm~\ref{algo_arc_restarting_for_proof}
		& $\beta=0$& 73 & 10 & 2 \\
		& $\beta=0.001$& 73 & 10 & 2\\
		& $\beta=0.5$& 72 & 10 & 3 \\
		& $\beta=0.999$& 73 & 10 & 3 \\
		& $\beta=1$& 73 & 10 & 3 \\
		\hline
		Algorithm~\ref{algo_arc_restarting_for_calculation}
		& $\beta=0.001$& 77 & 7 & 1 \\
		& $\beta=0.01$& 77 & 7 & 1 \\
		& $\beta=0.1$& 78 & 6 & 1 \\
		& $\beta=0.3$& 78 & 6 & 1 \\
		& $\beta=0.5$, Setting of $\beta_k$: \eqref{def_beta_for_proof}
		& 79 & 6 & 0 \\
		& $\beta=0.5$, Setting of $\beta_k$: \eqref{def_beta_for_calculation}
		& 78 & 5 & 2 \\
		& $\beta=0.7$& 78 & 6 & 1 \\
		& $\beta=0.9$& 79 & 5 & 1 \\
		& $\beta=0.99$& 77 & 6 & 2 \\
		& $\beta=0.999$& 78 & 6 & 1\\
		& $\beta=1$& 77 & 5 & 3 \\
		\hline
	\end{tabular}
\end{table}

From Table~\ref{table_comparison_solvable_problem_number},
we can see that
Algorithm~\ref{algo_arc_restarting_for_calculation} improves
the number of solvable problems from Algorithm~\ref{algo_arc_restarting_for_proof}.
When $\beta = 0.5$ and the setting of $\beta_k$ was \eqref{def_beta_for_proof},
Algorithms~\ref{algo_arc_restarting_for_proof} and \ref{algo_arc_restarting_for_calculation}
could not solve ten and six problems, respectively,
among the 85 problems due to the iteration limit.
As we will see soon in Figure~\ref{fig_comparison_by_algorithm},
Algorithm~\ref{algo_arc_restarting_for_calculation} can reduce the number of iterations compared to Algorithm~\ref{algo_arc_restarting_for_proof},
therefore, Algorithm~\ref{algo_arc_restarting_for_calculation} can reach more optimal solutions before the iteration limit as shown in Table~\ref{table_comparison_solvable_problem_number}.

In addition,
Table~\ref{table_comparison_solvable_problem_number} indicates
Algorithm~\ref{algo_arc_restarting_for_proof} with setting $\beta = 1$ does not affect the number of solvable problems remarkably,
despite the upper bounds on $\max \left\{\norm{\left(D^k\right)^{-1} \dot{z}},\norm{\left(D^k\right) \dot{s}}\right\}$
in Lemma~\ref{lemma_upper_norm_D_dot_with_x_s} diverges when we take $\beta \to 1$.
If $\beta=1$ in Algorithm~\ref{algo_arc_restarting_for_proof},
$x^k_i + \beta_k \delta(x^k)_i = 0$ may happen at some index $i$,
thus $(x^k + \beta_k \delta(x^k), \lambda^k, s^k) \notin \Neighborhood(\theta)$
holds and $z^k = x^k$ is selected in \eqref{def_restarted_point}.
As can be inferred from this,
the closer $\beta$ is to 1, the more likely $z^k=x^k$, and such an iteration is the same as the case of $\beta=0$.
Therefore, the divergence of $C_0$ is unlikely to occur in actual computation up to numerical errors.

The performance profile~\cite{Tits1996,gould2016note} on the number
of iterations in Algorithm~\ref{algo_arc_restarting_for_proof} with different $\beta$
is shown in Figure~\ref{fig_comparison_by_parameter_in_algo_proof}.
In the performance profiling,
the horizontal axis is a scaled value $\tau$ and
the vertical axis $P(r \le \tau)$ is the proportion of test problems.
For example, for the number of iterations,
$P(r \le \tau)$ corresponds
the percentage of test problems that were solved by less than a $\tau$ times the number of iterations of the best algorithm.
Simply speaking,
a good algorithm has a larger value $P(r \le \tau)$ from smaller $\tau$.
To output the performance profile, we used a Julia package~\cite{orban-benchmarkprofiles-2019}.
In Figure~\ref{fig_comparison_by_parameter_in_algo_proof},
the numbers of iterations with $\beta=0.999$ are the exactly same as those with $\beta=1$,
thus the plot for $\beta=0.999$ is hidden by that for  $\beta=1$ in the figure.

\begin{figure}[htb]
	\centering
	\includegraphics[scale=0.5]{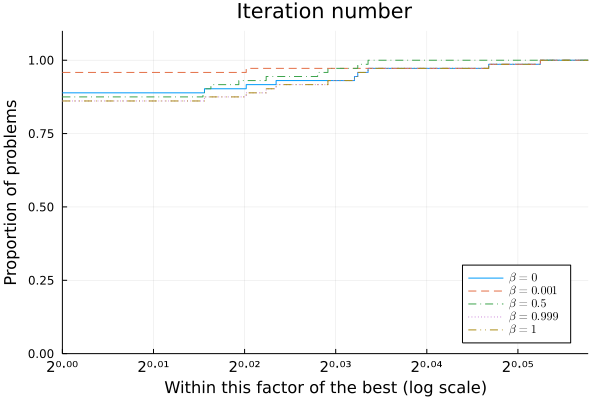}
	\caption{Performance profile of the number of iterations with different restarting parameters in Algorithm~\ref{algo_arc_restarting_for_proof}}
	\label{fig_comparison_by_parameter_in_algo_proof}
\end{figure}

Figure~\ref{fig_comparison_by_parameter_in_algo_proof}
shows that the momentum term improves the number of iterations when $\beta$ is 0.001.
In contrast, when $\beta$ is close to 1, 	$x^k + \beta_k \delta(x^k)$ tends to be out of the neighborhood $\Neighborhood(\theta)$, therefore, selecting $z^k = x^k$
by \eqref{def_restarted_point} in many iterations makes
the effect of the momentum term small.

Figure~\ref{fig_comparison_by_algorithm} compares
Algorithm~\ref{algo_arc_restarting_for_proof} and \ref{algo_arc_restarting_for_calculation} under the setting of $\beta = 0.001$.
Figure~\ref{fig_comparison_by_algorithm} indicates that
the result of Algorithm~\ref{algo_arc_restarting_for_calculation} is significantly better than that of Algorithm~\ref{algo_arc_restarting_for_proof},
since Algorithm~\ref{algo_arc_restarting_for_calculation} uses the momentum term through all the iterations.

\begin{figure}[htb]
	\centering
	\includegraphics[scale=0.5]{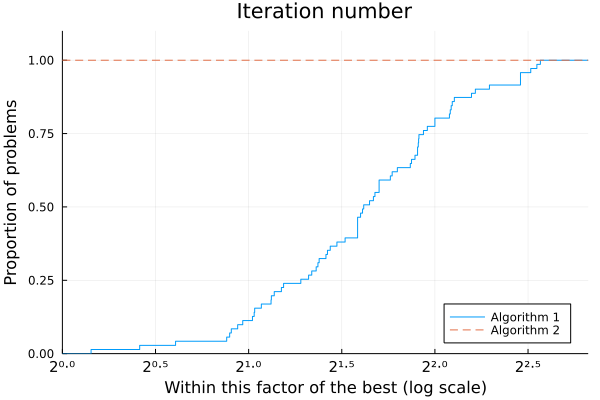}
	\caption{
		Performance profile of the number of iterations
		in Algorithms~\ref{algo_arc_restarting_for_proof} and \ref{algo_arc_restarting_for_calculation} with $\beta = 0.001$
	}
	\label{fig_comparison_by_algorithm}
\end{figure}

From the above result,
we use the results of Algorithm~\ref{algo_arc_restarting_for_calculation}
in the following numerical experiments instead of Algorithm~\ref{algo_arc_restarting_for_proof}.

\subsubsection{Effect of the choice of the weight of the momentum term}
\label{subsection_comparison_setting_of_beta_k}
Theorem~\ref{theorem_residual_decreasing} indicates that
$\beta_k$  by \eqref{def_beta_for_proof} can guarantee the decreasing of  $r_b(x^k)$,
thus
the number of solvable problems with \eqref{def_beta_for_proof} should be larger than \eqref{def_beta_for_calculation}.
In contrast,
the numerical result in Table~\ref{table_comparison_solvable_problem_number} shows that
Algorithm~\ref{algo_arc_restarting_for_calculation} with $\beta=0.5$ and the formula \eqref{def_beta_for_proof} is 79,
while the algorithm with $\beta=0.5$ and the formula \eqref{def_beta_for_calculation}
is 78, therefore, the difference in solvable problems between \eqref{def_beta_for_proof}
and \eqref{def_beta_for_calculation} is only one.

In Figure~\ref{fig_comparison_trajectories_for_KB2},
we plot the residuals and the duality measure
on a test instance KB2 solved with  $\beta_k$ in \eqref{def_beta_for_proof} and \eqref{def_beta_for_calculation}.
When we use  \eqref{def_beta_for_proof},
the primal residual $\norm{r_b(x^k)}_\infty$ decreases monotonically
in the first six iterations to a level of $10^{-8}$.
In contrast, when we use \eqref{def_beta_for_calculation},
 $\norm{r_b(x^k)}_\infty$ slightly increases at the fifth iteration.
However, both
\eqref{def_beta_for_proof} and \eqref{def_beta_for_calculation}
decrease $\norm{r_b(x^k)}_\infty$ sufficiently in 25 iterations.

\begin{figure}[ht]
	\begin{tabular}{cc}
		\begin{minipage}{0.55\textwidth}
			\includegraphics[width=\columnwidth]{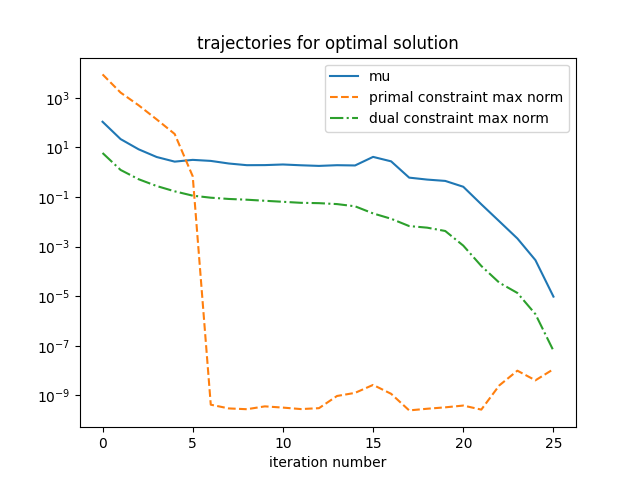}
		\end{minipage}
		\begin{minipage}{0.55\textwidth}
			\includegraphics[width=\columnwidth]{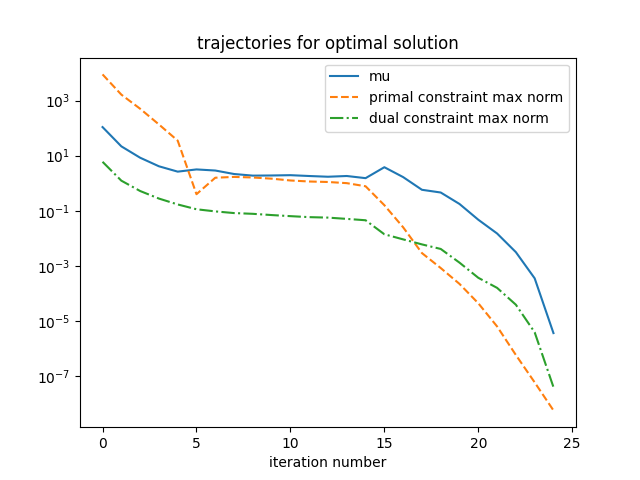}
		\end{minipage}
	\end{tabular}
	\caption{Trajectories for $\norm{r_b(x^k)}_\infty$. The left is the result with the setting of \eqref{def_beta_for_proof}, and the right is with \eqref{def_beta_for_calculation}.}
	\label{fig_comparison_trajectories_for_KB2}
\end{figure}

Figure~\ref{fig_comparison_by_setting_beta_k} shows the performance profile on the number of iterations, and we can see that \eqref{def_beta_for_calculation} is better.
\begin{figure}[ht]
	\centering
	\includegraphics[scale=0.4]{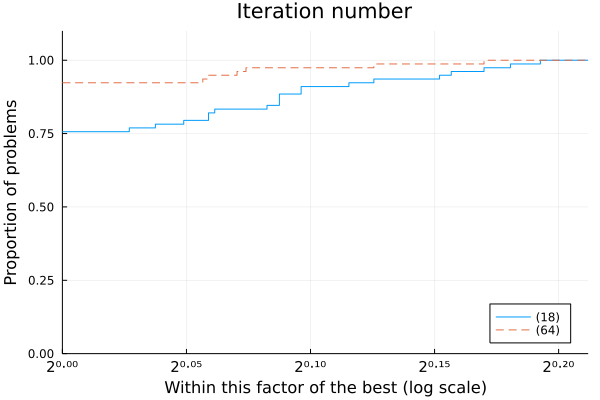}
	\caption{Performance profile of the number of iterations with different setting of $\beta_k$}
	\label{fig_comparison_by_setting_beta_k}
\end{figure}

\subsubsection{Numerical sensitivity of the restarting parameter}
\label{subsection_comparison_beta}
Here, we investigate influence of the restarting parameter $\beta$.

Figure~\ref{fig_comparison_by_parameter} shows the performance profiling on
the number of iterations of Algorithm~\ref{algo_arc_restarting_for_calculation}
with different $\beta$ from $0.001$ to $1$.
We can observe there that  larger $\beta$ tends to lead to
a slightly less number of iterations, and
this implies that a momentum term is effective to reduce the iterations.

In the following, we fix $\beta = 0.9$ in Algorithm~\ref{algo_arc_restarting_for_calculation}
which solves more problems in Table~\ref{table_comparison_solvable_problem_number},
and we compare Algorithm~\ref{algo_arc_restarting_for_calculation} with the existing methods.

\begin{figure}[ht]
	\centering
	\includegraphics[width=\linewidth]{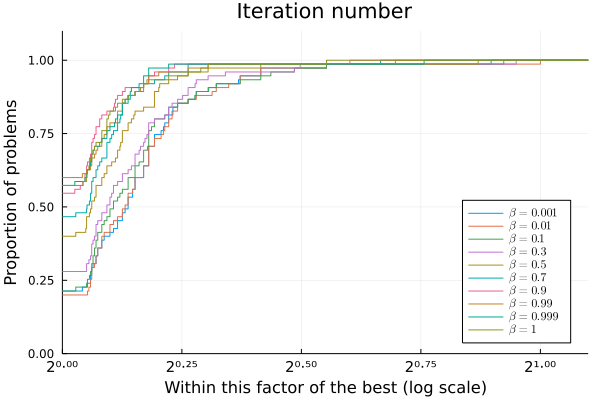}
	\caption{Performance profile of the number of iterations with different restarting parameters in Algorithm~\ref{algo_arc_restarting_for_calculation}}
	\label{fig_comparison_by_parameter}
\end{figure}

\subsubsection{Comparison with existing methods}
\label{subsection_comparison_with_existing_methods}
Figure~\ref{fig_comparison_with_existing_methods} shows a performance profile that compares Algorithm~\ref{algo_arc_restarting_for_calculation} and the two existing methods,
Arc-search~\cite{Yang2017} and Line-search~\cite{Mehrotra1992}.
The results in Figure~\ref{fig_comparison_with_existing_methods} indicates that
Algorithm~\ref{algo_arc_restarting_for_calculation} performs better than the two existing methods in terms of the number of iterations.

\begin{figure}[ht]
	\centering
	\includegraphics[scale=0.5]{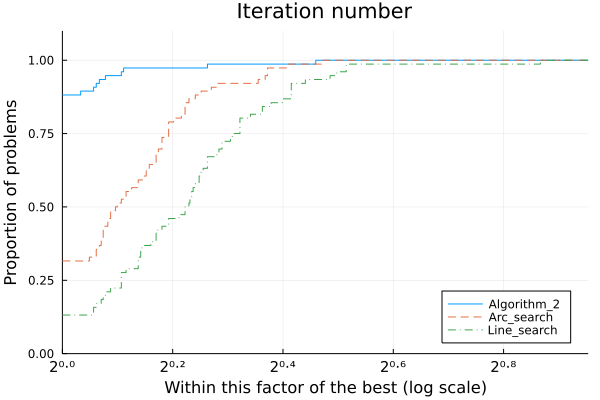}
	\caption{Performance profile of the number of iterations with Algorithm~\ref{algo_arc_restarting_for_calculation} and existing methods}
	\label{fig_comparison_with_existing_methods}
\end{figure}

The detailed numerical results of the three methods are reported in Table~\ref{table_results_for_comparison} of Appendix.
Among 76 problems in which all three methods found optimal solutions,
Algorithm~\ref{algo_arc_restarting_for_calculation} reduces the number of iterations in 45 problems compared to Arc-search and Line-search
(or finishes the iterations within at most the same numbers of iterations in 67 problems).

We also discuss the computational time.
Since the results of small problems are too short,
we calculate average computational times
on problems for which all the three methods spent 30 seconds or longer.
The average computational times are
252.51 in Algorithm~\ref{algo_arc_restarting_for_calculation},
272.73 in Arc-search, and
264.04 in Line-search,
therefore, Algorithm~\ref{algo_arc_restarting_for_calculation} reduce the average time by about 5.4\%.
\section{Conclusion}
\label{section_conclusion}
In this work,
we proposed an iterative method for LP problems by integrating
Nesterov's restarting strategy into the arc-search interior-point method.
In particular, we approximated the central path with an ellipsoidal arc
based on the point $z^k$ that is computed from $x^k$
with the momentum term $\delta(x^k)$.
By keeping all the iterations in the neighborhood $\Neighborhood(\theta)$,
we established the convergence of a generated sequence of the proposed method to an optimal solution and showed
the proposed method achieves the polynomial-time computational complexity.
In order to improve numerical performance,
we also proposed the modified method with a concept of the Mehrotra type interior-point method.
From the numerical experiments with the Netlib benchmark problems,
we observed that the modified method achieved better performance than existing
arc-search and line-search interior-point methods.

As a future direction,
we consider that Nesterov's restarting strategy in arc-search interior-point methods can be extended to more general types of problems,
such as second-order cone programming, semidefinite programming, linear constrained quadratic programming, and more general conic programming.
In particular,
there is room for further discussions on how the restarting strategy behaves in nonlinear-constrained cases.

In addition,
the multiple centrality correction discussed
in~\cite{gondzio1996multiple,Colombo2008} may also be combined to improve the numerical performance,
since the step size can be larger when the iteration points are closer to the central path.

\section*{Appendix}
\subsection*{Details on numerical results}
Table~\ref{table_results_for_comparison} reports the numerical results in Section~\ref{subsection_comparison_with_existing_methods}.
The first column of the table is the problem name, while
the second and the third are  the variable size $n$ and
the number of constraints $m$, respectively.
The fourth to ninth  columns reports
the numbers of iterations and the computation time of the methods.
A mark `-' indicates the unsolved problem due to (i), (ii) or (iii) in Section~\ref{section_stopping_criteria}.
or that the step size $\alpha_k$ is too small.
The results in bold face indicate the best results among the three methods.

\begin{longtable}{lrr|rr|rr|rr}
	\caption{Numerical results on Algorithm~\ref{algo_arc_restarting_for_calculation},
	Arc-search~\cite{Yang2017} and Line-search~\cite{Mehrotra1992}}
	\label{table_results_for_comparison} \\
	\hline
	problem name & $n$ & $m$ &
	\multicolumn{2}{c|}{Algorithm~\ref{algo_arc_restarting_for_calculation}} &
	\multicolumn{2}{c|}{Arc-search~\cite{Yang2017}} & \multicolumn{2}{c}{Line-search~\cite{Mehrotra1992}} \\
	& & &
	Itr. & Time & Itr. & Time & Itr. & Time \\
	\hline \hline
	25FV47 & 1835 & 780 & \textbf{23} & 24.13 & 24 & \textbf{23.33} & 27 & 28.97 \\
	ADLITTLE & 137 & 55 & \textbf{10} & \textbf{0.08} & 11 & \textbf{0.08} & 13 & \textbf{0.08} \\
	AFIRO & 51 & 27 & \textbf{7} & \textbf{0.03} & 8 & \textbf{0.03} & 9 & \textbf{0.03} \\
	AGG & 476 & 388 & \textbf{31} & 14.74 & 40 & 16.23 & \textbf{31} & \textbf{12.03} \\
	AGG2 & 755 & 514 & \textbf{21} & \textbf{10.46} & 25 & 11.35 & 25 & 11.98 \\
	AGG3 & 755 & 514 & \textbf{23} & 11.39 & \textbf{23} & \textbf{9.81} & 28 & 13.56 \\
	BANDM & 371 & 216 & \textbf{15} & 4.76 & 17 & \textbf{3.89} & 19 & 5.43 \\
	BEACONFD & 167 & 77 & \textbf{7} & \textbf{0.07} & 8 & 0.08 & 10 & \textbf{0.07} \\
	BNL1 & 1496 & 611 & 44 & 32.82 & 55 & 37.29 & \textbf{43} & \textbf{31.19} \\
	BNL2 & 4335 & 2209 & \textbf{28} & \textbf{184.42} & 31 & 202.27 & 33 & 217.67 \\
	BOEING1 & 856 & 490 & \textbf{20} & \textbf{10.92} & 22 & 10.95 & 28 & 12.72 \\
	BOEING2 & 333 & 194 & \textbf{21} & \textbf{4.81} & 22 & 7.35 & 26 & 8.08 \\
	BORE3D & 128 & 71 & 19 & 0.16 & \textbf{18} & 0.14 & 19 & \textbf{0.12} \\
	BRANDY & 227 & 122 & \textbf{19} & 4.75 & 20 & 3.52 & 22 & \textbf{2.78} \\
	CAPRI & 480 & 275 & \textbf{23} & 7.96 & \textbf{23} & 6.18 & 24 & \textbf{6.1} \\
	CRE-A & 6997 & 3299 & \textbf{25} & 569.99 & \textbf{25} & \textbf{557.27} & 26 & 576.95 \\
	CRE-C & 5684 & 2647 & \textbf{26} & \textbf{324.25} & 29 & 361.66 & 28 & 342.03 \\
	CYCLE & 3123 & 1763 & \textbf{24} & \textbf{71.98} & \textbf{24} & 73.74 & 26 & 77.14 \\
	CZPROB & 2786 & 678 & \textbf{29} & \textbf{45.91} & 34 & 55.68 & 32 & 50.73 \\
	D2Q06C & 5807 & 2147 & \textbf{29} & \textbf{325.99} & 30 & 343.07 & 34 & 388.62 \\
	D6CUBE & 6183 & 403 & \textbf{20} & 162.82 & \textbf{20} & 158.59 & \textbf{20} & \textbf{156.83} \\
	DEGEN3 & 2604 & 1503 & \textbf{15} & \textbf{30.84} & 17 & 33.85 & 18 & 36.45 \\
	ETAMACRO & 771 & 436 & 24 & 9.46 & 24 & 10.27 & \textbf{23} & \textbf{9.31} \\
	FFFFF800 & 990 & 486 & - & - & - & - & \textbf{59} & \textbf{31.06} \\
	FINNIS & 976 & 465 & \textbf{17} & \textbf{7.87} & 18 & 7.99 & 21 & \textbf{7.87} \\
	FIT1D & 2075 & 1050 & 27 & 30.89 & \textbf{25} & \textbf{30.11} & 30 & 36.29 \\
	FIT1P & 2076 & 1026 & \textbf{17} & \textbf{20.05} & 22 & 26.22 & 31 & 35.83 \\
	GANGES & 1753 & 1356 & \textbf{14} & \textbf{16.1} & 17 & 20.52 & 20 & 22.77 \\
	GREENBEA & 4536 & 2187 & - & - & - & - & \textbf{53} & \textbf{378.63} \\
	GREENBEB & 4524 & 2182 & \textbf{31} & \textbf{218.74} & 36 & 253.29 & 37 & 259.54 \\
	ISRAEL & 309 & 167 & 22 & 5.77 & \textbf{21} & \textbf{4.62} & 24 & 4.69 \\
	KB2 & 77 & 52 & \textbf{24} & 0.15 & \textbf{24} & 0.14 & \textbf{24} & \textbf{0.11} \\
	LOTFI & 357 & 144 & \textbf{15} & 3.46 & \textbf{15} & 3.34 & 17 & \textbf{3.07} \\
	MAROS & 1510 & 713 & \textbf{29} & \textbf{22.03} & 31 & 22.53 & 32 & 25.23 \\
	MAROS-R7 & 7448 & 2156 & \textbf{12} & \textbf{237.6} & 13 & 256.78 & 16 & 314.89 \\
	MODSZK1 & 1621 & 685 & \textbf{22} & 19.02 & 24 & \textbf{18.93} & 29 & 24.68 \\
	PDS-02 & 9253 & 4523 & \textbf{20} & \textbf{1031.95} & - & - & 22 & 1112.86 \\
	PEROLD & 1650 & 764 & \textbf{34} & \textbf{28.35} & 41 & 32.11 & \textbf{34} & 28.93 \\
	PILOT & 5348 & 2173 & 56 & 538.86 & 72 & 694.07 & \textbf{52} & \textbf{502.17} \\
	PILOT-WE & 3145 & 951 & \textbf{38} & \textbf{81.64} & - & - & - & - \\
	PILOT4 & 1393 & 628 & \textbf{94} & \textbf{70.9} & - & - & - & - \\
	PILOT87 & 7776 & 3365 & \textbf{39} & \textbf{1111.14} & 44 & 1214.68 & 42 & 1150.4 \\
	PILOTNOV & 2315 & 1040 & \textbf{19} & \textbf{27.18} & 20 & 27.63 & 20 & 27.59 \\
	RECIPELP & 203 & 117 & \textbf{9} & 2.3 & \textbf{9} & 1.94 & 11 & \textbf{0.35} \\
	SC105 & 162 & 104 & \textbf{9} & 2.85 & 10 & 2.86 & 11 & \textbf{0.49} \\
	SC205 & 315 & 203 & \textbf{11} & 3.38 & 13 & 3.94 & 13 & \textbf{1.05} \\
	SC50A & 77 & 49 & \textbf{8} & 0.05 & \textbf{8} & \textbf{0.04} & 9 & \textbf{0.04} \\
	SC50B & 76 & 48 & \textbf{7} & 0.05 & \textbf{7} & \textbf{0.04} & 8 & \textbf{0.04} \\
	SCAGR25 & 572 & 372 & \textbf{16} & 6.72 & 17 & \textbf{6.56} & 19 & 6.7 \\
	SCAGR7 & 158 & 102 & \textbf{12} & \textbf{0.74} & \textbf{12} & 2.99 & 15 & 2.55 \\
	SCFXM1 & 567 & 304 & \textbf{17} & 6.36 & 18 & 6.63 & 20 & \textbf{4.64} \\
	SCFXM2 & 1134 & 608 & \textbf{20} & 13.79 & 21 & \textbf{12.64} & 22 & 15.85 \\
	SCFXM3 & 1701 & 912 & \textbf{20} & 18.69 & 21 & \textbf{18.53} & 22 & 19.72 \\
	SCRS8 & 1202 & 424 & \textbf{19} & \textbf{10.22} & 20 & 12.23 & 21 & 11.19 \\
	SCSD1 & 760 & 77 & \textbf{8} & \textbf{0.42} & 9 & 0.46 & 9 & 0.5 \\
	SCSD6 & 1350 & 147 & \textbf{10} & 5.21 & \textbf{10} & 5.28 & 12 & \textbf{5.05} \\
	SCSD8 & 2750 & 397 & \textbf{9} & \textbf{11.75} & 10 & 22.83 & 12 & 15.18 \\
	SCTAP1 & 660 & 300 & \textbf{16} & 6.43 & 18 & 5.71 & 17 & \textbf{4.26} \\
	SCTAP2 & 2500 & 1090 & \textbf{13} & \textbf{20.46} & \textbf{13} & 21.9 & 14 & 24.09 \\
	SCTAP3 & 3340 & 1480 & \textbf{13} & \textbf{42.16} & 14 & 44.53 & 14 & 43.32 \\
	SEBA & 1345 & 890 & \textbf{17} & \textbf{13.09} & \textbf{17} & 13.17 & 19 & 14.95 \\
	SHARE1B & 243 & 107 & 26 & 6.63 & \textbf{25} & 6.8 & 26 & \textbf{1.19} \\
	SHARE2B & 161 & 95 & \textbf{12} & \textbf{0.12} & 14 & 0.15 & 15 & \textbf{0.12} \\
	SHELL & 1520 & 529 & \textbf{20} & \textbf{14.62} & \textbf{20} & 15.01 & 21 & 15.82 \\
	SHIP04L & 1963 & 325 & \textbf{12} & 11.44 & \textbf{12} & \textbf{11.08} & 15 & 13.6 \\
	SHIP04S & 1323 & 241 & \textbf{12} & 6.41 & 13 & \textbf{6.34} & 15 & 6.68 \\
	SHIP08L & 3225 & 526 & \textbf{14} & \textbf{26.66} & 15 & 28.56 & 17 & 33.0 \\
	SHIP08S & 1688 & 322 & \textbf{12} & \textbf{9.31} & 14 & 10.78 & 16 & 11.26 \\
	SHIP12L & 4274 & 664 & \textbf{15} & \textbf{54.67} & 17 & 60.97 & 20 & 72.1 \\
	SHIP12S & 2038 & 390 & \textbf{14} & \textbf{12.91} & 16 & 14.57 & 19 & 16.11 \\
	STAIR & 387 & 205 & \textbf{17} & \textbf{4.49} & 22 & 7.96 & 20 & 5.74 \\
	STANDATA & 1273 & 395 & \textbf{12} & \textbf{7.49} & \textbf{12} & 9.54 & 14 & 8.31 \\
	STANDGUB & 1273 & 395 & \textbf{12} & \textbf{7.23} & \textbf{12} & 7.66 & 14 & 9.29 \\
	STANDMPS & 1273 & 503 & \textbf{14} & \textbf{8.59} & 16 & 10.32 & 18 & 12.87 \\
	STOCFOR1 & 139 & 91 & \textbf{17} & \textbf{0.16} & \textbf{17} & \textbf{0.16} & 24 & 0.18 \\
	STOCFOR2 & 2868 & 1980 & \textbf{28} & 83.62 & \textbf{28} & 85.24 & \textbf{28} & \textbf{83.08} \\
	TRUSS & 8806 & 1000 & \textbf{16} & \textbf{375.51} & 17 & 385.02 & 19 & 430.57 \\
	TUFF & 619 & 307 & \textbf{24} & \textbf{7.42} & 27 & 11.82 & 27 & 8.43 \\
	VTP-BASE & 155 & 103 & \textbf{20} & 3.71 & 23 & \textbf{2.26} & 22 & 7.82 \\
	WOOD1P & 1802 & 171 & 36 & 29.3 & 35 & 25.88 & \textbf{30} & \textbf{20.18} \\
	WOODW & 5368 & 712 & 55 & 355.49 & 53 & 333.81 & \textbf{40} & \textbf{246.85} \\
	\hline
\end{longtable}

\bibliographystyle{abbrv}
\bibliography{scholar}

\end{document}